\documentclass{article}

\usepackage{arxiv}

\usepackage[utf8]{inputenc} 
\usepackage[T1]{fontenc}    
\usepackage{hyperref}       
\usepackage{url}            
\usepackage{booktabs}       
\usepackage{amsfonts}       
\usepackage{nicefrac}       
\usepackage{microtype}      
\usepackage{graphicx}
\usepackage{amsmath}
\usepackage{array}
\usepackage{doi}
\usepackage{amsthm}
\usepackage{dsfont}
\usepackage{mathtools}
\usepackage{tcolorbox}
\usepackage{amsmath}
\usepackage{physics}
\usepackage{bbm}
\usepackage{wasysym}
\usepackage{caption}
\usepackage{subcaption}

\usepackage{nicefrac} 
\usepackage{doi}
\usepackage{physics}
\usepackage{subcaption}

\newcommand{\mi}{\mathrm{i}}
\newcommand{\bx}{\boldsymbol{x}}
\DeclareMathOperator{\thh}{th}

\newcommand{\B}{\mathbf{B}}
\renewcommand{\H}{\mathbf{H}}
\newcommand{\E}{\mathbf{E}}
\newcommand{\F}{\mathbf{F}}
\newcommand{\C}{\mathbf{C}}
\renewcommand{\c}{c}
\newcommand{\D}{\mathbf{D}}
\newcommand{\X}{\mathbf{X}}

\newcommand{\M}{\mathbf{M}}
\newcommand{\W}{\mathbf{W}}
\newcommand{\N}{\mathbb{N}}
\newcommand{\R}{\mathbb{R}}
\newcommand{\Z}{\mathbb{Z}}
\newcommand{\bZ}{\mathbf{Z}}
\newcommand{\K}{\mathbf{K}}

\usepackage{xcolor}

\newcommand{\G}{\mathbf{G}}

\newcommand{\g}{g}

\usepackage{amsmath}
\usepackage{mathtools}
\usepackage{comment}

\newcommand\coolover[2]{\mathrlap{\smash{\overbrace{\phantom{%
    \begin{matrix} #2 \end{matrix}}}^{\mbox{$#1$}}}}#2}

\newcommand\coolunder[2]{\mathrlap{\smash{\underbrace{\phantom{%
    \begin{matrix} #2 \end{matrix}}}_{\mbox{$#1$}}}}#2}

\newcommand\coolleftbrace[2]{%
#1\left\{\vphantom{\begin{matrix} #2 \end{matrix}}\right.}

\newcommand\coolrightbrace[2]{%
\left.\vphantom{\begin{matrix} #1 \end{matrix}}\right\}#2}

\usepackage{amsmath}
\usepackage{tikz}
\usetikzlibrary{calc}

\newcommand{\tikzmark}[1]{\tikz[overlay,remember picture] \node (#1) {};}
\newcommand{\DrawBox}[4][]{%
    \tikz[overlay,remember picture]{%
        \coordinate (TopLeft)     at ($(#2)+(-0.2em,0.9em)$);
        \coordinate (BottomRight) at ($(#3)+(0.2em,-0.3em)$);
        \path (TopLeft); \pgfgetlastxy{\XCoord}{\IgnoreCoord};
        \path (BottomRight); \pgfgetlastxy{\IgnoreCoord}{\YCoord};
        \coordinate (LabelPoint) at ($(\XCoord,\YCoord)!0.5!(BottomRight)$);
        \draw [red,#1] (TopLeft) rectangle (BottomRight);
        \node [below, #1, fill=none, fill opacity=1] at (LabelPoint) {#4};
    }
}

\usepackage{tikz}
\newcommand{\circleasterisk}{%
  \begin{tikzpicture}[baseline=-0.5ex]
    \draw (0,0) circle (0.5em);
    \node at (0,0) {$*$};
  \end{tikzpicture}%
}
\usepackage{tikz}
\newcommand{\squareasterisk}{%
  \begin{tikzpicture}[baseline=-0.5ex]
    \draw (-0.1,-0.1) rectangle (0.6em, 0.6em);
    \node at (0.06,0.06) {$*$};
  \end{tikzpicture}%
}

\numberwithin{equation}{section}

\newtheorem{theorem}{Theorem}[section]
\newtheorem{assumption}[theorem]{Assumption}
\newtheorem{lemma}[theorem]{Lemma}
\newtheorem{proposition}[theorem]{Proposition}
\newtheorem{corollary}[theorem]{Corollary}
\newtheorem{remark}[theorem]{Remark}

\newtheorem{example}[theorem]{Example}

\title{Stability of conforming space--time \\ isogeometric methods for the wave equation}

\author{\href{https://orcid.org/0000-0002-2577-1421}{\includegraphics[scale=0.06]{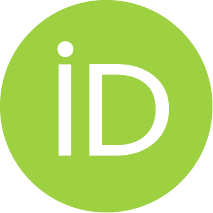}\hspace{1mm}Matteo ~Ferrari} \\
	Faculty of Mathematics \\
    University of Vienna\\
    1090 Vienna, Austria \\
	\texttt{matteo.ferrari@univie.ac.at} \\
\And
\href{https://orcid.org/0009-0005-7783-4438}{\includegraphics[scale=0.06]{orcid.pdf}\hspace{1mm} Sara ~Fraschini} \\
	Faculty of Mathematics \\
    University of Vienna \\
    1090 Vienna, Austria \\
	\texttt{sara.fraschini@univie.ac.at} \\
}

\renewcommand{\undertitle}{}
\renewcommand{\shorttitle}{}

\hypersetup{
pdftitle={stability space-time iso},
pdfsubject={q-bio.NC, q-bio.QM},
pdfauthor={Ferrari M.~Fraschini S.},
pdfkeywords={spline space-time},
}

\begin{document}
\maketitle

\begin{abstract}
We consider a family of conforming space--time finite element discretizations for the wave equation based on splines of maximal regularity in time. Traditional techniques may require a CFL condition to guarantee stability. Recent works by O. Steinbach and M. Zank (2018), and S. Fraschini, G. Loli, A. Moiola, and G. Sangalli (2023), have introduced unconditionally stable schemes by adding non-consistent penalty terms to the underlying bilinear form. Stability and error analysis have been carried out for lowest order discrete spaces. While higher order methods have shown promising properties through numerical testing, their rigorous analysis was still missing. In this paper, we address this stability analysis by studying the properties of the condition number of a family of matrices associated with the time discretization. For each spline order, we derive explicit estimates of both the CFL condition required in the unstabilized case and the penalty term that minimises the consistency error in the stabilized case. Numerical tests confirm the sharpness of our results.
\end{abstract}

\section{Introduction}
The foundations of space--time Finite Element Methods (FEMs) for wave propagation problems have been laid in the seminal paper \cite{HH1988}. Unlike standard time-stepping techniques, space--time methods involve the simultaneous discretization of space and time variables, yielding numerical solutions available at all time instants. The popularity of these space--time schemes has been promoted in the last decades by recent advancements in computer technology, along with the remarkably properties of space--time discretizations. These include effective handling of moving boundaries, high order accuracy both in space and in time, multilevel preconditioning, local mesh refinement and adaptivity, and space--time parallelization.

The majority of space--time methods for the wave equation belong to the discontinuous Galerkin family (see, e.g. \cite{ImbertMoiolaStocker2023,MoiolaPerugia2018,MR2005} and references therein). In contrast, there has been much less work dedicated to space--time conforming Galerkin methods, for which only a few works are available. Among these, we mention the recent works \cite{BM2023, FLMS2023, LSZ2023, SZ2019, SZ2020}. It has been shown in \cite{F2021, SZ2019} that stability of conforming space--time methods for the second order acoustic wave equation may be subject to a CFL condition. This limitation has been overcome in \cite{SZ2019, Z2021}, where an unconditionally stable space--time piecewise continuous FEM has been designed by stabilizing the bilinear form coming from integration by parts. In \cite{F2021, FLMS2023}, generalizing the idea from \cite{SZ2019}, a high order unconditionally stable space--time isogeometric method has been devised.

The isogeometric method, also known as Isogeometric Analysis (IgA), has been introduced in~\cite{HCB2005} as an advancement from classical FEMs, aiming to streamline the integration of computer-aided design with numerical simulations. Essentially, the isoparametric concept is invoked, since IgA employs spline functions, or their generalizations, to both parametrize the computational domain and approximate the differential equation of interest. Leveraging the smoothness of splines leads to increased accuracy per degree of freedom when compared to classical $C^0$ piecewise polynomial discretizations (see, e.g. \cite{BS2019,EBBH2009}), and superior spectral approximation properties \cite{HRS2008}. 

In this paper, we consider the space--time isogeometric method introduced in \cite{FLMS2023}. We provide the first theoretical justification of its unconditional stability. As in \cite{SZ2019}, we carry out our analysis focusing on an ordinary differential equation strongly related to the time part of the acoustic wave equation. Our theoretical results are derived by exploiting the algebraic structure of the matrices involved, and are aimed at studying the properties of their condition numbers. The analysis we perform combines two main tools: properties of the so-called symbols \cite{GaroniSpeleersEkstromRealiSerraCapizzanoHughes2019} of the matrices associated with IgA discretizations, and the behaviour of the condition number of families of Toeplitz band matrices characterized in \cite{AmodioBrugnano1996}. All at once we are able to obtain, for each spline degree, sharp estimates on the CFL condition required in the unstabilized case, and to derive which penalty term guarantees unconditional stability while minimising the consistency error. Moreover, we apply our analysis technique to different choices for the penalty term, and show that, among those, the one proposed in \cite{FLMS2023} is the best one.

The paper is structured as follows. In Section \ref{sec:2}, we revisit the framework considered in \cite{F2021,FLMS2023,SZ2019}, detailing the variational setting and the proposed stabilized bilinear forms. In Section \ref{sec:3}, we provide an overview of splines with maximal regularity and of the properties of the associated Galerkin matrices. In Section \ref{sec:4}, we recall and extend results on the conditioning behaviour of families of Toeplitz band matrices, with emphasis on those exhibiting particular symmetries in their coefficients. In Section \ref{sec:5}, we present the proof of our main result on the CFL condition and the optimal penalty term. In Section \ref{sec:6}, we consider and analyze other possible stabilizations terms. Finally, in Section \ref{sec:7}, we present numerical tests that demonstrate that our estimates are sharp.

\section{Model problem and stabilized variational formulations} \label{sec:2}

We consider the linear acoustic wave equation defined on the space--time cylinder $Q = \Omega \times (0, T )$,
\begin{equation} \label{eq:21}
	\begin{cases}
		\partial_{t}^2 U(\bx,t) - \Delta_{\bx} U(\bx,t) = F(\bx,t), & (\bx,t) \in Q, 
		\\ U(\bx, t) = 0, & (\bx, t) \in \partial\Omega \times (0, T),
		\\ U(\bx, 0) = 0, \quad \partial_t U(\bx, t)_{|_{t=0}} = 0, & \bx \in \Omega,
	\end{cases}
\end{equation}
where $\Omega \subset \R^d$ ($d \in \N$) is a bounded Lipschitz domain, $T > 0$ is a finite time, and $F \in L^2(Q)$ is a given source term. 

To recall the space--time variational formulations considered in~\cite{FLMS2023, SZ2019},  we introduce the Hilbert spaces
\begin{equation*}
	\begin{aligned}
		H^{1,1}_{0;0,\bullet}(Q)&=L^2\big(0,T;H^1_{0}(\Omega)\big)\cap H^1_{0,\bullet}\big(0,T;L^2(\Omega)\big), \quad
		H^{1,1}_{0;\bullet,0}(Q)&=L^2\big(0,T;H^1_{0}(\Omega)\big)\cap H^1_{\bullet,0}\big(0,T;L^2(\Omega)\big),
	\end{aligned}
\end{equation*}
where we denote $H^1_{0}(\Omega) = \left\{w \in H^1(\Omega) : \ w_{|_{\partial\Omega}}=0 \right\}$, and the Sobolev-Bochner spaces
\begin{align*}	 
	H^1_{0,\bullet}\big(0,T;L^2(\Omega)\big) & =\{W \in H^1\big(0,T;L^2(\Omega)\big): \ W(\cdot,0)=0\}, \\
	H^1_{\bullet,0}\big(0,T;L^2(\Omega)\big) & = \{W \in H^1\big(0,T;L^2(\Omega)\big): \ W(\cdot,T)=0\}.
\end{align*}
After integration by parts, the space--time variational formulation for~\eqref{eq:21} reads as: find $U \in H^{1,1}_{0;0,\bullet}(Q)$ such that
\begin{equation} \label{eq:22}
	-(\partial_t U, \partial_t V)_{L^2(Q)} + (\nabla_{\bx} U, \nabla_{\bx} V)_{L^2(Q)}  = (F, V)_{L^2(Q)}, \, \text{ for all } V \in H^{1,1}_{0;\bullet,0}(Q),
\end{equation}
where $(\cdot,\cdot )_{L^2(Q)}$ denotes the inner product in $L^2(Q)$. Note that the initial condition $U(\bx,0) = 0$ and the homogeneous Dirichlet condition are considered in a strong sense, while $\partial_t U(\bx,t)_{|_{t=0}}\hspace{-0.05cm}= 0$ is incorporated in the variational formulation.

Exploiting the Fourier expansion of the trial and tests functions, it has been shown (see, e.g. \cite{Z2020}) that there exists a unique solution $U \in H^{1,1}_{0;0,\bullet}(Q)$ of the variational formulation \eqref{eq:22}, stable with respect to the source term $F$. In greater details, let $\{\psi_\ell\}_\ell$ be an orthogonal basis in $H^1_{0,\bullet}(0,T)= \{w \in H^1(0,T): \ w(0)=0\}$ composed by the eigenfunctions of the second time derivative with zero initial condition, and let $\{\lambda_j\}_j$ be the eigenfunctions of the space Laplacian with Dirichlet boundary conditions. Any $U \in H^{1,1}_{0;0,\bullet}(Q)$ admits the representation
\begin{equation*}
    U(\bx,t) = \sum_j \sum_\ell U_{j,\ell}\psi_\ell(t)\lambda_j(\bx) = \sum_j u_j(t)\lambda_j(\bx), \quad u_j(t)=\sum_\ell U_{j,\ell}\psi_\ell(t),
\end{equation*}
with suitable coefficients $\{U_{j,\ell}\}_{j,\ell} \subset \R$. Choosing in \eqref{eq:22} a test function $v(t)\lambda_j(\bx)$ with $v \in H^1_{\bullet,0}(0,T) = \{w \in H^1(0,T): \ w(T)=0\}$, it follows that the variational problem \eqref{eq:22} is equivalent to finding, for each $j$, the coefficient functions $u_j \in H^1_{0,\bullet}(0,T)$ such that
\begin{equation*}
    -(\partial_t u_j, \partial_t v)_{L^2(0,T)} + \mu_j ( u_j, v)_{L^2(0,T)} = ( f_j, v) _{L^2(0,T)}, \quad \text{ for all } v \in H^1_{\bullet,0}(0,T),
\end{equation*}
where $\{\mu_j\}_j$ is the non-decreasing, positive and divergent sequence of the eigenvalues of the space Dirichlet Laplacian, and $f_j(t) = (F(\cdot,t), \lambda_j)_{L^2(\Omega)}$. 

Our main focus lies in providing a finite element discretization for the initial value problem: find $u \in H_{0,\bullet}^1(0,T)$ such that
\begin{equation} \label{eq:2.3}
	a_\mu(u,v) = (f,v)_{L^2(0,T)} \quad \text{for all~} \, v \in H^1_{\bullet,0}(0,T),
\end{equation}
with $f \in L^2(0,T)$, and the bilinear form $a_\mu : H^1_{0,\bullet}(0,T) \times H^1_{\bullet,0}(0,T) \to \R$  given by
\begin{equation} \label{eq:2.4}
	a_\mu(u,v) =-(\partial_tu, \partial_tv)_{L^2(0,T)}+ \mu (u,v)_{L^2(0,T)}, \quad \text{for~} \mu > 0.
\end{equation}
If this method is stable independently of the mesh parameter that characterizes the finite element subspace and $\mu$, then we expect unconditional stability for the associated space--time discretization of the wave problem \eqref{eq:22}.

In \cite[Theorem 4.1]{SZ2020} it has been shown unique solvability of \eqref{eq:2.3} by a compact perturbation argument. However, this result is not sufficient to guarantee unconditional stability (with respect to $\mu$) for each numerical discretization of \eqref{eq:2.3}. It has been observed in \cite[Remark 4.9]{SZ2020} that a piecewise-linear continuous Galerkin-Petrov FEM discretization of \eqref{eq:2.3}, with a uniform mesh size $h$, is stable if and only if $h^2 \mu < 12$. This mesh condition turns out to be a CFL of the type $h_t < C_\Omega h_x$ in the associated space--time variational formulation for the wave equation, with $C_\Omega > 0$ depending on the domain $\Omega$, and $h_t$ and $h_x$ the temporal and spatial mesh parameters, respectively. To overcome this mesh condition, the bilinear form \eqref{eq:2.4} is stabilized in \cite{SZ2019}. Specifically, an unconditionally stable Galerkin-Petrov discretization is proposed by perturbing the bilinear form \eqref{eq:2.4} with a suitable penalty term \cite[Corollary 17.1]{SZ2019}, i.e., considering
\begin{equation*}
	a_{\mu,h}(u_h,v_h) =-(\partial_t u_h, \partial_t v_h)_{L^2(0,T)}+ \mu (u_h,v_h)_{L^2(0,T)} -\frac{\mu h^2}{12} ( \partial_t u_h, \partial_t v_h )_{L^2(0,T)}.
\end{equation*}
For all $h$ and $\mu$, the perturbed bilinear form $a_{\mu,h}$ satisfies a discrete inf–sup condition \cite[Lemma 17.6]{SZ2019}, and ensures optimal error estimates in $H^1$ and $L^2$ norms. An extension of this stabilization to high order Galerkin-Petrov $C^0$ finite elements has been proposed and numerically verified in \cite{Z2021}.

In \cite{F2021,FLMS2023}, in the same spirit, a stabilized numerical scheme for isogeometric discretizations with maximal regularity splines has been proposed adding a suitable non-consistent penalty term. For a spline space of degree $p$ with maximal regularity $C^{p-1}$, the proposed stabilized bilinear form reads as
\begin{equation} \label{eq:2.5}
\begin{aligned}
    a^p_{\mu,h}(u_h,v_h) = - (\partial_t u_h, \partial_t v_h)_{L^2(0,T)} & + \mu (u_h,v_h)_{L^2(0,T)} + \mu \delta h^{2p} (\partial_t^p u_h, \partial_t^p v_h)_{L^2(0,T)},
    \end{aligned}
\end{equation} 
for a suitable parameter $\delta \le 0$. The stabilized bilinear form \eqref{eq:2.5} is then employed in the space--time finite element discretization for the wave equation \eqref{eq:21} defining
\begin{equation} \label{eq:2.6}
\begin{aligned}
    a^p_h(U_h,V_h) =- (\partial_t U_h, \partial_t V_h)_{L^2(Q)} & + (\nabla_{\bx} U_h, \nabla_{\bx} V_h)_{L^2(Q)} + \delta h_t^{2p} (\nabla_{\bx} \partial_t^p U_h, \nabla_{\bx} \partial_t^p V_h)_{L^2(Q)}.
    \end{aligned}
\end{equation} 
Even if the proposed stabilized isogeometric formulation of~\cite{F2021,FLMS2023,FMS2022} is supported by a wide range of numerical experiments highlighting its stability and approximation properties, rigorous theoretical results are still missing. In the next sections, we address this theoretical analysis. Precisely, in Theorem \ref{th:510} below, we show that there exist two computable quantities $\delta_p <0$ and $\rho_p >0$ such that the Galerkin method associated with the bilinear form \eqref{eq:2.5} is stable, in a sense that will be specified later,
\begin{itemize}
	\item for all $h$ and $\mu$, if and only if $\delta \le \delta_p$,
	\item if and only if $\mu h^2 \le \rho_p$, when $\delta=0$.  
\end{itemize}

\section{Maximal regularity splines and associated matrices} \label{sec:3}
In this section, we describe the maximal regularity spline spaces and leverage certain properties of the matrices representing the bilinear form \eqref{eq:2.5}.

For $N \in \N$, we introduce the uniform mesh parameter $h = T/N$ and the break points $t_\ell = \ell h$, for $\ell = 0,\ldots,N$. For a positive integer $p \ge 1$, let $\{\xi_j\}_{j=0}^{2p+N} \subset [0,T]$ be an open knot vector. According to the Cox-De Boor recursion formulas, spline basis functions are piecewise polynomials defined recursively as 
\begin{align} \label{eq:3.1}
	\varphi^p_j(t) =
	\begin{cases}
        		\dfrac{t-\xi_{j}}{\xi_{j+p}-\xi_{j}}\varphi_j^{p-1}(t)   +\dfrac{\xi_{j+p+1}-t}{\xi_{j+p+1}-\xi_{j+1}}\varphi_{j+1}^{p-1}(t) & { \text{if }} t \in [\xi_j,\xi_{j+p+1}), 
        		\\ 0  & \text{otherwise},
	\end{cases}
\end{align}
with $j = 0,\ldots,N+p-1$, beginning with $\varphi^0_j(t) = 1$, if $t \in [\xi_{j},\xi_{j+1})$, and $\varphi^0_j(t)=0$ otherwise. The knot vector that guarantees maximal $C^{p-1}$ regularity is built as
\begin{equation*}
\begin{aligned}
	\xi_0 = \xi_1 = \ldots = \xi_p = t_0=0, \quad \xi_{p+1}=t_1, \quad \quad \ldots \quad \quad \xi_{p+N-1}=t_{N-1}, \quad \xi_{p+N} = \ldots = \xi_{2p+N} = t_{N} = T.
\end{aligned}
\end{equation*}
The space of splines generated by $\{\varphi_j^p\}_{j=0}^{N+p-1}$ is denoted by $S_h^p(0,T)$, while the trial and test spaces for the discrete variational formulation \eqref{eq:2.3} are, respectively,
\begin{equation*}
    S^p_{h,0,\bullet}(0,T) = S_h^p(0,T) \cap H_{0,\bullet}^1(0,T) \quad\quad \text{and} \quad\quad S^p_{h,\bullet,0}(0,T) = S_h^p(0,T) \cap H_{\bullet,0}^1(0,T).
\end{equation*}
From the $p^{\thh}$ basis function to the $(N-1)^{\thh}$, each spline maintains the same shape but is supported in different intervals. We say that $\varphi_j^p$ is central if it satisfies $p \le j \le N-1$. The support of the central spline $\varphi_j^p$ is $[t_{j-p},t_{j+1}]$.
\begin{figure}[h!]
    \centering
    \begin{subfigure}[b]{0.3\textwidth}
    \centering
         \includegraphics[width=\textwidth]{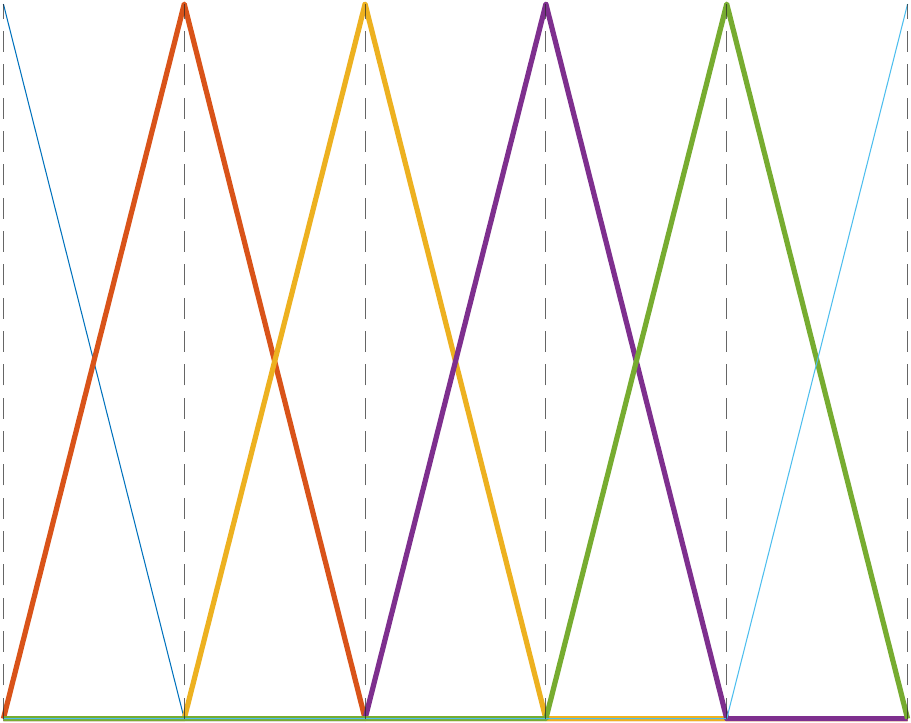}
         \caption{$p=1$}
     \end{subfigure}
     \hfill
    \begin{subfigure}[b]{0.3\textwidth}
    \centering
         \includegraphics[width=\textwidth]{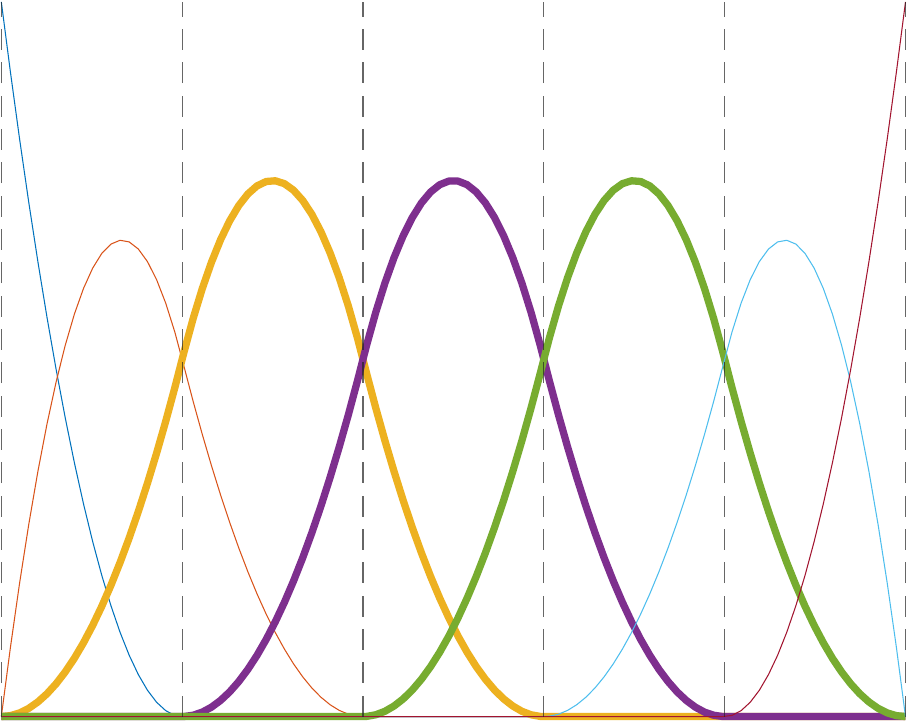}
         \caption{$p=2$}
     \end{subfigure}
     \hfill
     \begin{subfigure}[b]{0.3\textwidth}
    \centering
         \includegraphics[width=\textwidth]{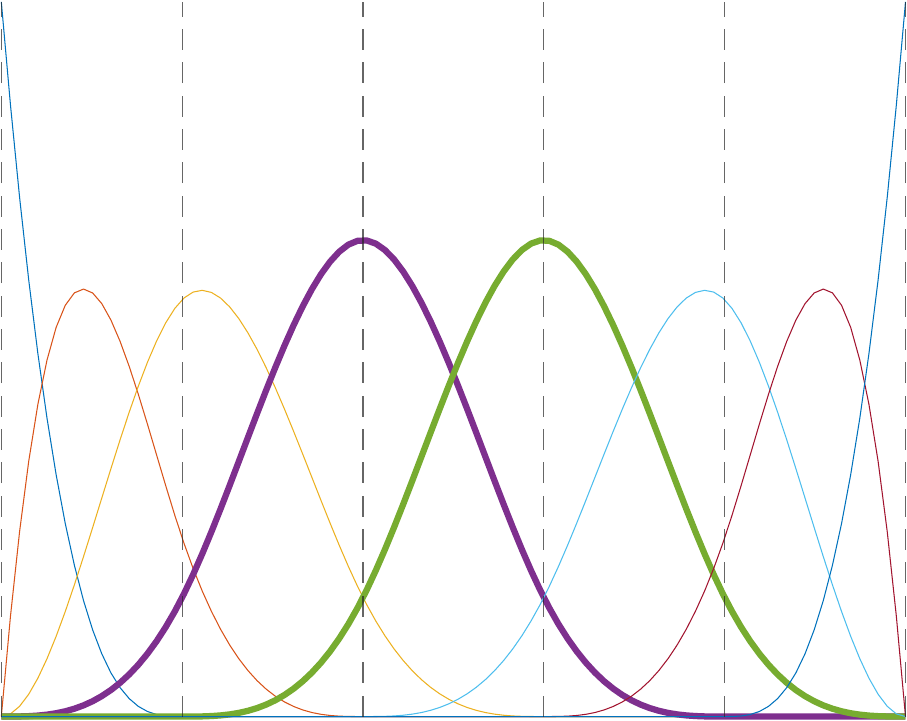}
         \caption{$p=3$}
     \end{subfigure}
\caption{The spline basis of $S_{h}^p(0,1)$ with $N = 5$. The break points are indicated by vertical dashed lines, and central splines are highlighted with thicker lines.}
\end{figure}

We now examine the matrix representation of the stabilized bilinear form $a_{\mu,h}^p$ \eqref{eq:2.5} with respect to the discrete spaces $S^p_{h,0,\bullet}(0,T) \times S^p_{h,\bullet,0}(0,T)$. Let us define the three matrices
\begin{equation} \label{eq:3.2}
\begin{aligned}
	\M^p_{h}[\ell,j] = \int_0^T & \varphi_j^p(t)  \varphi^p_{\ell-1}(t) \dd t, \quad \B^p_{h}[\ell,j] = \int_0^T \partial_t \varphi_j^p(t) \partial_t \varphi^p_{\ell-1}(t) \dd t, 
 \\ & \D^p_{h}[\ell,j] = \int_0^T \partial_t^p \varphi_j^p(t) \partial_t^p \varphi^p_{\ell-1}(t) \dd t
\end{aligned}
\end{equation}
with $\ell, j = 1,\ldots$ $n$, where we denote by $n$ the size of these three matrices, precisely
\begin{equation} \label{eq:3.3}
    n = N+p-1.
\end{equation}
\begin{example}
Here we explicitly write the entries of the matrices $\M_{h}^2, \B_{h}^2, \D_{h}^2 \in \R^{(N+1)\times (N+1)}$ to show their structure. With the exception of $5$ entries at the top-left and $5$ at the bottom-right corners, they are Toeplitz band matrices:
\begin{align*}
\hspace{-1.3cm}
\M_{h}^2 & = \frac{h}{120} 
\setcounter{MaxMatrixCols}{20}
    \begin{pmatrix} 
            14 & 2  \\
            40 & 25 & 1 \\
            25 & 66 & 26 & 1 \\
            1 & 26 & 66 & 26 & 1 \\
            & \ddots & \ddots & \ddots & \ddots & \ddots \\
            & &  1 & 26 & 66 & 26 & 1 \\
            & & & 1 & 26 & 66 & 25 & 2 \\
            & & & & 1 & 25 & 40 & 14 \\
    \end{pmatrix}_{(N+1) \times (N+1)}\hspace{-1.95cm},
    \\
\B_{h}^2 & = \frac{1}{6h} 
\setcounter{MaxMatrixCols}{20}
    \begin{pmatrix} 
            -6 & -2  \\ 
            8 & -1 & -1 \\
            -1 & 6 & -2 & -1 \\
            -1 & -2 & 6 & -2 & -1 \\
            & \ddots & \ddots & \ddots & \ddots & \ddots \\
            & & -1 & -2 & 6 & -2 & -1 \\
            & & & -1 & -2 & 6 & -1 & -2 \\
            & & & & -1 & -1 & 8 & -6 \\
    \end{pmatrix}_{(N+1) \times (N+1)} \hspace{-1.95cm},
\end{align*}
and 
\begin{equation*}
\D_{h}^2 = 
\setcounter{MaxMatrixCols}{20}
    \frac{1}{h^3} \begin{pmatrix} 
            -6 & 2  \\
            10 & -5 & 1 \\
            -5 & 6 & -4 & 1 \\
            1 & -4 & 6 & -4 & 1 \\
            & \ddots & \ddots & \ddots & \ddots & \ddots \\
            & & 1 & -4 & 6 &-4 & 1 \\
            & & & 1 & -4 & 6 & -5 & 2 \\
            & & & & 1 & -5 & 10 & -6 \\
    \end{pmatrix}_{(N+1) \times (N+1)}\hspace{-1.8cm}.
\end{equation*}
\end{example}
Let $\Phi_p$ denote the cardinal B-spline of degree $p$, namely the spline function defined over the uniform knot sequence $\{0,\ldots,p+1\}$ as
\begin{equation*}
\Phi_p(s) = \frac{s}{p} \Phi_{p-1}(s) + \frac{p+1-s}{p} \Phi_{p-1}(s-1), \quad p \ge 1
\end{equation*}
beginning with $\Phi_0(t) = 1$ if $t \in [0,1)$ and $\Phi_0(t) =0$ otherwise. Observe that the central splines are essentially cardinal splines that takes into account the scaling parameter $h$ and a translation. Specifically, for $j = p, \ldots, N-1$, each central spline $\varphi_j^p(t)$ can be expressed as $\Phi_p(t/h - j + p)$. Consequently, we can represent the matrices in \eqref{eq:3.2} as
\begin{equation} \label{eq:3.4}
\begin{aligned}
	\M^p_{h}[\ell,j] = \int_0^T \Phi_p(t/h-j+p) \Phi_p(t/h-\ell+1+p) \dd t= h \int_\R \Phi_p(s-j) \Phi_p(s-\ell+1) \dd s,
\end{aligned}
\end{equation}
for $j = p, \ldots, N-1$ and $\ell = p+1 \ldots, N$, and similarly
\begin{equation} \label{eq:3.5}
\begin{aligned}
	 \B^p_{h}[\ell,j] = \frac{1}{h} \int_\R \partial_s \Phi_p(s-j) \partial_s \Phi_p(s-\ell+1) \dd s, \quad 
	   \D^p_{h}[\ell,j] = \frac{1}{h^{2p-1}} \int_\R \partial_s^p \Phi_p(s-j) \partial_s^p \Phi_p(s-\ell+1) \dd s.
	 \end{aligned}
\end{equation}
In light of the previous expressions, we outline  some considerations on the scale and structures of these matrices that are consequences of well-known results on cardinal B-splines (see \cite[Section 4.2]{Chui1992} and \cite{GaroniManniPelosiSerraCapizzanoSpeleers2014}).
\begin{proposition} \label{prop:32}
For all $p \ge 1$, let $\M_h^p, \B_h^p$ and $\D_h^p$ be defined in \eqref{eq:3.2}. Then, the following properties hold true.
\begin{enumerate}
\item The entries within matrices $\frac{1}{h} \M_h^p$, $h \B_h^p$ and $h^{2p-1} \D_h^p$ are independent of the mesh parameter $h$.
\item The matrices $\M_h^p$, $\B_h^p$ and $\D_h^p$ are persymmetric, i.e., they are symmetric about their northeast-southwest diagonal.
\item The matrices $\M_h^1$, $\B_h^1$ and $\D_h^1$ are lower triangular Toeplitz band matrices with three non-zero diagonals. For $p>1$, with the exception of $2p^2-3$ entries located at the top-left and bottom-right corners, the matrices $\M_h^p$, $\B_h^p$, and $\D_h^p$ exhibit a Toeplitz band structure. In particular, in the top-left corner, precisely the non-zero entries of the first $p$ rows and the first $p-1$ columns do not respect the Toeplitz structure, with the exception of the entries in position $(p,2p-1)$ and $(2p,p-1)$. The precise structure of that block is as follows
\begin{equation*}
        \vphantom{
        \begin{matrix}
            \overbrace{XYZ}^{\mbox{$R$}}\\ \\ \\ \\ \\ \\ \\ \\
            \underbrace{pqr}_{\mbox{$S$}} \\
        \end{matrix}}%
    \begin{matrix}
        \vphantom{a}
        \coolleftbrace{p+1}{* \\ \vdots \\ * \\ * \vspace{0.5cm}} \vspace{0.1cm}\\ 
        \coolleftbrace{p-1}{* \\ \vdots \\ *}
    \end{matrix}%
    \begin{pmatrix}
        \coolover{p-1}{*     & \ldots &     * \hspace{0.1cm}}& \coolover{p}{* & \phantom{***} & \phantom{*}  & \phantom{\circleasterisk}} 
        \\        \vdots &        & \vdots               &         \vdots & \ddots           & 
        \\        *      & \ldots & *                    & *              &  \ldots          &  *           &
        \\         *     & \ldots &     *                &              * &  \ldots          &  *           & \circleasterisk &
        \\        *      & \ldots & *                    &                &                  &              &
        \\         *     & \ldots &     *                &                &                  &              &              &
        \\               & \ddots & \vdots               &                &                  & 
        \\               &        & \circleasterisk      &                &                  &  
        \end{pmatrix}.%
\end{equation*}
\item In their purely Toeplitz band parts, the matrices exhibit symmetry in the entries with respect to the first lower codiagonal. Specifically, the central non-vanishing elements can be expressed as
\begin{equation} \label{eq:36}
	\begin{aligned}
		\frac{1}{h} \M^p_{h}[\ell,\ell-1 \pm j] = \Phi_{2p+1}(p+1-j), & 
		\quad h \B^p_{h}[\ell,\ell-1 \pm j] = - \partial_s^2 \Phi_{2p+1}(p+1-j),
		\\ h^{2p-1} \D^p_h[\ell,\ell-1 \pm j] & = (-1)^p \partial_s^{2p} \Phi_{2p+1}(p+1-j),
	\end{aligned}
\end{equation}
for $j=0,\ldots,p$ and $\ell = 2p+1,\ldots,N-p$.
\end{enumerate}
\end{proposition}
\begin{proof}
The first three properties readily follow from the definition of the matrices \eqref{eq:3.2}, and the B-splines definition \eqref{eq:3.1}. On the other hand, representations \eqref{eq:36} are consequences of \eqref{eq:3.4} and \eqref{eq:3.5}, along with the characterization for inner products of derivatives of cardinal B-splines \cite[Lemma 4]{GaroniManniPelosiSerraCapizzanoSpeleers2014}
\begin{equation} \label{eq:37}
	\int_{\R} \partial^{k_1}_s \Phi_p(s) \partial^{k_2}_s \Phi_p(s+j) \dd s = (-1)^{k_2} \partial^{(k_1+k_2)}_s \Phi_{2p+1}(p+1-j)
\end{equation}
and their symmetry property \cite[Lemma 3]{GaroniManniPelosiSerraCapizzanoSpeleers2014}
\begin{equation} \label{eq:38}
	\partial_s^k \Phi_{2p+1}(p+1+s) = (-1)^k \partial_s^k \Phi_{2p+1}(p+1-s).
\end{equation}
\end{proof}
\begin{remark}
The indices of the first and last rows representing integrals involving only central splines are the $(2p+1)^{\thh}$ and $(N-p)^{\thh}$, respectively. There exists at least one row of that type if and only if $N-p \ge 2p+1$, i.e., $N \ge 3p+1$.
\end{remark}
The system matrix representing the bilinear form $a_{\mu,h}^p$, as defined in \eqref{eq:2.5}, with respect to the spaces $S^p_{h,0,\bullet}(0,T) \times S^p_{h,\bullet,0}(0,T)$ is given by 
\begin{equation} \label{eq:39}
    \K_{\mu,h}^p(\delta) = -\B_h^p + \mu \M_h^p + \mu \delta h^{2p} \D_h^p.
\end{equation}
We make the following fundamental assumption about the system matrices, which we expect to be true for all $p \in \N$ and $\delta \le 0$, but which is currently based only on our numerical experience.
\begin{assumption} \label{assu:1}
We assume that $p \in \N$ and $\delta \le 0$ are such that $\K_{\mu,h}^p(\delta)$ is nonsingular for all $\mu, h>0$.
\end{assumption}
\begin{remark} For $p=1$, Assumption \ref{assu:1} is readily verified since $\K_{\mu,h}^1(\delta)$ is a triangular Toeplitz matrix with entries on the diagonal $1/h + (\mu h)/6 - \mu  \delta h$.
\end{remark}
Upon looking at the scaled matrix $h\K_{\mu,h}^p(\delta)$, in accordance with Proposition \ref{prop:32}, its dependence on $\mu$ and $h$ is solely through the product $\mu h^2$. Consequently, we introduce the following definitions, with $n = N+p-1$ as in \eqref{eq:3.3},
\begin{equation} \label{eq:310}
    \rho = \mu h^2\,, \quad \K_n^p(\rho,\delta) = h\K_{\mu,h}^p(\delta),
\end{equation}
where $\K_n^p(\rho,\delta) \in \R^{n \times n}$. By virtue of \eqref{eq:36}, the non-zero entries corresponding to the symmetric Toeplitz band part of $\K_n^p(\rho,\delta)$ can be expressed as
\begin{equation} \label{eq:311}
\begin{aligned}
		\K_n^p(\rho,\delta)[\ell,\ell-1 \pm j] = \partial_s^2 \Phi_{2p+1}(p+1-j)
         + \rho \bigl( \Phi_{2p+1}(p+1-j) + \delta (-1)^p \partial_s^{2p} \Phi_{2p+1}(p+1-j) \bigr)
\end{aligned}		
\end{equation}
for $j=0,\ldots,p$ and $\ell=2p+1,\ldots,n-2p+1$.

In the next sections, for fixed $p \ge 1$ and under Assumption \ref{assu:1}, we investigate the spectral properties of the family of matrices $\{\K^p_n(\rho,\delta)\}_n$ with respect to the size $n$ and to the parameters $\rho$ and $\delta$. We derive necessary and sufficient conditions that ensure the well-conditioning (in a sense that will be specified later) of these matrix families in two distinct cases:
\begin{itemize}
\item we find $\delta_p \in \R$ such that $\{\K_n^p(\rho,\delta)\}_n$ is well-conditioned for all $\rho > 0$ if and only if $\delta \le \delta_p$;
\item we find $\rho_p \in \R$ such that $\{\K_n^p(\rho,0)\}_n$ is well-conditioned if and only if $\rho \le \rho_p$.
\end{itemize}

\section{Conditioning of Toeplitz band matrices} \label{sec:4}
In this section, after reviewing some results related to the conditioning of families of Toeplitz band matrices, we extend them to certain families of perturbed Toeplitz band matrices (with detailed proof in Appendix \ref{app:A}). Then, we derive an easily verifiable necessary condition for the well-conditioning of these families when the corresponding matrices exhibit symmetry with respect to the first lower codiagonal (Section \ref{sec:42}).

A family of nonsingular matrices $\{\C_n\}_n$, with $\C_n \in \R^{n\times n}$, is said to be \textit{well-conditioned} if the condition numbers ${\kappa(\C_n)}$ are uniformly bounded with respect to $n$. It is said to be \textit{weakly well-conditioned} if ${\kappa(\C_n)}$ grows as a power of $n$. Here and in the subsequent sections, the condition numbers are computed using either the $\| \cdot \|_1$ or $\| \cdot \|_\infty$ matrix norms.

Given the coefficients $\{\c_{-m},\ldots,\c_0,\ldots,\c_k\}$, consider the family of Toeplitz band matrices $\{ \C_n \}_{n}$,
\begin{equation} \label{eq:41}
\C_n =  
\setcounter{MaxMatrixCols}{20}
    \begin{pmatrix} 
            \c_0 & \ldots & \c_k & &  \\ 
            \vdots & \ddots & & \ddots \\
            \c_{-m} & & \ddots & & \c_k \\
            & \ddots & & \ddots & \vdots \\
            &  & \c_{-m} & \ldots & \c_0 \\
    \end{pmatrix}_{n \times n} \hspace{-0.6cm}.
\end{equation}
We associate to this family the polynomial $q^{\C} \in \mathbb{P}_{m+k}(\R)$
\begin{equation} \label{eq:42}
    q^{\C}(z) = \sum_{j=-m}^k \c_j z^{m+j}.
\end{equation}
Without loss of generality, we assume that $\c_{-m} \c_k \ne 0$. In \cite{AmodioBrugnano1996}, the authors show that the well-conditioning (or weakly well-conditioning) of the family $\{\C_n\}_{n}$ depends on the location of zeros of $q^{\C}$.

We define a polynomial as being of type $(s,u,\ell)$ if it has $s$ zeros with modulus smaller than $1$, $u$ zeros with unit modulus, and $\ell$ zeros with modulus larger than $1$. 
\begin{theorem}{\cite[Theorem 3]{AmodioBrugnano1996}} \label{th:41}
Let $\{\C_n\}_{n}$ represent a family of nonsingular Toeplitz band matrices, and let $q^{\C}$ in \eqref{eq:42} be the associated polynomial. Then, the family $\{\C_n\}_{n}$ exhibits the following behaviour:
\begin{itemize}
\item it is well-conditioned if $q^{\C}$ is of type $(m,0,k)$;
\item it is weakly well-conditioned if $q^{\C}$ is of type $(m_1,m_2,k)$ or $(m,k_1,k_2)$, where $m_1+m_2=m$ and $k_1+k_2=k$. In this scenario, $\kappa(\C_n)$ grows at most as $\mathcal{O}(n^\eta)$, with $\eta$ representing the highest multiplicity among the zeros of unit modulus of $q^{\C}$.
\end{itemize}
Additionally, if $q^{\C}$ is of type $(q_1,q_2,q_3)$ for $q_1,q_2$ and $q_3$ such that $q_1+q_3>0$ and simultaneously $q_1 \ne m$ and $q_3 \ne k$, then $\kappa(\C_n)$ grows exponentially as $n$ increases.
\end{theorem}
\begin{remark} Weakly well-conditioning may be true for families of type $(0,m+k,0)$. Consider  indeed the family $\{\C_n\}_n$ of symmetric tridiagonal Toeplitz matrices with $c_{-1} = 1, c_{0} = 2$, and $c_1 = 1$. The polynomial $q^{\C}(z) = 1+2z+z^2$ is of type $(0,2,0)$, and the eigenvalues of matrices of this family are explicitly known (see, e.g. \cite[Pag. 59]{Smith1985})
\begin{equation*}
    \left\{\lambda_j = 2+2 \cos\left( \frac{j \pi}{n+1} \right) \quad j = 1,\ldots,n\right\}.
\end{equation*}
From the latter expression we readily obtain that $\kappa_2(\C_n) = \mathcal{O}(n^2)$.
\end{remark}
\begin{remark}
Property that $q^{\C}$ is of type $(m_1,m_2,k)$ or $(m,k_1,k_2)$ , where $m_1+m_2=m$ and $k_1+k_2=k$ actually implies nonsingularity, for $n$ sufficiently large, of the elements of the family $\{\C_n\}_n$ (see \cite[Remark 2]{AmodioBrugnano1996}).
\end{remark}
\begin{remark}
Note that similar results can be obtained in the $2$-norm by recalling that the following inequality is valid (see, for example, \cite[Equations 2.3.11 and 2.3.12]{GolubVanLoan1996})
\begin{equation*}
    \frac{1}{\sqrt{n}} \| \C \|_1 \le \| \C \|_2 \le \sqrt{n} \| \C \|_1, \quad \text{for all~} \C \in \R^{n \times n}.
\end{equation*}
In particular, weakly well-conditioning with respect to a given norm $\| \cdot \|_1$, $\| \cdot \|_\infty$ or $\| \cdot \|_2$ is equivalent to weakly well-conditioning in any other of the two norms.
\end{remark}
In the next result we extend Theorem \ref{th:41} to certain families of perturbed Toeplitz band matrices that include $\{\K_n^p(\rho,\delta)\}_n$ defined in \eqref{eq:310} under Assumption \ref{assu:1}. More precisely, Theorem \ref{th:41} applies to families of nonsingular perturbed Toeplitz band matrices, where the perturbations are restricted to the top-left and bottom-right corners in blocks of size $(m+k) \times (m+k)$.
\begin{theorem} \label{th:45}
Let $\{\C_n\}_n$ be a family of nonsingular Toeplitz band matrices as in \eqref{eq:41}. Consider a family $\{\widehat{\C}_n\}_n$ consisting of nonsingular matrices that are perturbations of $\{\C_n\}_n$, where the perturbations of each $\widehat{\C}_n$ are restricted to the top-left and bottom-right blocks of size $(m+k) \times (m+k)$ with entries independent on $n$, depicted as
\begin{equation} \label{eq:43}
    \vphantom{
    \begin{matrix}
            \overbrace{XYZ}^{\mbox{$R$}} \\ \\ \\ \\ \\
            \underbrace{pqr}_{\mbox{$S$}} \\
        \end{matrix}}
    \begin{matrix}
        \coolleftbrace{m}{* \\ \vspace{0.15cm} \\ * } \vspace{0.1cm}\\ 
        \coolleftbrace{k}{* \\ \vspace{0.15cm} \\ *}
    \end{matrix}%
    \hspace{-0.15cm}
    \begin{pmatrix}
        \coolover{k}{*     & \ldots &     * \hspace{0.2cm}}& \coolover{m}{\squareasterisk & \phantom{**}  & \phantom{**}} 
        \\        \vdots &        & \vdots               &         \vdots & \ddots           & 
        \\         *     & \ldots &     *                &              * &  \ldots          &  \squareasterisk           
        \\         \squareasterisk     & \ldots &     *                &                &                  &                            
        \\               & \ddots & \vdots               &                &                  & 
        \\               &        & \squareasterisk                    &                &                  &  
        \end{pmatrix}, \hspace{0.2cm} \quad 
         \vphantom{
    \begin{matrix}
            \overbrace{XYZ}^{\mbox{$R$}}\\ \\ \\ \\ \\ \\ \\ \\
            \underbrace{pqr}_{\mbox{$S$}} \\
        \end{matrix}}%
    \begin{pmatrix}
         &        &                &        \squareasterisk  &            & 
        \\              &  &                     &   \vdots            &  \ddots          &             
        \\              &                &                  & *     & \ldots &     \squareasterisk                             
        \\             \squareasterisk  & \ldots & *   &       *     &    \ldots            &                 *  
        \\               &  \ddots & \vdots                    &   \vdots             &              &  \vdots
        \\ \coolunder{k}{ \phantom{**}  & \phantom{**} &     \squareasterisk } & \coolunder{m}{* & \ldots & *} 
        \end{pmatrix}
        \hspace{-0.15cm}
    \begin{matrix}
        \coolrightbrace{* \\ \vspace{0.15cm} \\ * }{m} \vspace{0.1cm}\\ 
        \coolrightbrace{* \\ \vspace{0.15cm} \\ * }{k} \vspace{0.1cm}
    \end{matrix}%
\end{equation}
where the matrix on the left represents the top-left perturbation block, and the matrix on the right the bottom-right perturbation block. Let us further assume that the matrices $\widehat \C_n$ have entries in the two outer codiagonals (the $k^{\thh}$ and the $(-m)^{\thh}$, i.e., the entries in the square boxes of the two perturbed blocks) all different from zero. Then, the families $\{\C_n\}_n$ and $\{\widehat{\C}_n\}_n$ have the same conditioning behaviour.
\end{theorem}
The proof of Theorem \ref{th:45} is mainly based on the techniques developed in \cite{AmodioBrugnano1996}, and it is reported in Appendix \ref{app:A}.
\begin{remark}
Finding a matrix criterion that guarantees the nonsingularity of the matrices $\widehat \C_n$ under the weakly well-conditioning of the family $\{\C_n\}_n$ should be interesting. A necessary condition is clearly the nonsingularity of the two perturbed blocks of size $(m+k) \times (m+k)$ in \eqref{eq:43}, but this is not sufficient (see Example \ref{ex:A7}). A sufficient condition, however difficult to verify in our case, is given in Remark \ref{rem:A6}, the proof of which is deduced from the proof of Theorem \ref{th:45} and \cite[Remark 2]{AmodioBrugnano1996}. At present, to apply Theorem \ref{th:45} to the families $\{\K_n^p(\rho,\delta)\}_n$ in \eqref{eq:310} we have to assume the nonsingularity of their elements as in Assumption \ref{assu:1}.
\end{remark}
Thanks to Proposition \ref{prop:32}, the family of perturbed Toeplitz band matrices $\{\K_n^p(\rho,\delta)\}_n$ satisfies the structure hypothesis of Theorem \ref{th:45}. In Appendix \ref{app:B} we verify up to $p=8$ that the entries in the outer codiagonals of their perturbation blocks are non-zero, and we expect the latter to be true for all $p \ge 9$ as well. Since this property can be verified a priori for larger $p$, unlike Assumption \ref{assu:1} which is much deeper, we will henceforth tacitly assume that it is satisfied.

\subsection{Conditioning of Toeplitz band matrices with symmetries} \label{sec:42}

For a given $p \ge 1$ and a set of $p+1$ real numbers $\{\g_j\}_{j=0}^p$ with $\g_0 \ne 0$, we consider the family of Toeplitz band matrices denoted by $\{\G_n^p\}_n$, with
\begin{equation} \label{eq:44}
	\G_n^p =  
	\setcounter{MaxMatrixCols}{20}
	\begin{pmatrix} 
		\g_{p-1} & \g_{p-2} & \dots & \g_1 & \g_0 & & &
		\\ \g_p & \ & \ & \  & \ & \ddots
		\\ \g_{p-1} & \ & \ & \ & \ & \ & \ddots
		\\ \vdots & \ & \ & \ & \ & \ & & \g_0
		\\ \g_1 & \ & \ & \ & \ & \ & & \g_1
		\\ \g_0 & \ & \ & \ & \ & \ & & \vdots
		\\ \ & \ddots & \ & \ & \ & \ & & \g_{p-2}
		\\ \ & \ & \g_0 & \g_1 & \dots & \g_{p-1} & \g_p & \g_{p-1}
	\end{pmatrix}_{n \times n} \hspace{-0.6cm}.
\end{equation}
These matrices align with the notations established in Theorem \ref{th:41}, with $m=p+1$ and $k=p-1$. Consequently, investigating the conditioning behaviour of such matrices entails analyzing the associated polynomial of degree $2p$
\begin{equation} \label{eq:45}
	q^{\G^p}(z) = \g_0 + \g_1 z + \ldots + \g_p z^p + \ldots + \g_1 z^{2p-1} + \g_0 z^{2p}.
\end{equation}
Leveraging the specific structure of the polynomial $q^{\G^p}$, we deduce the following result.
\begin{lemma} \label{lem:47}
Let the family of matrices $\{\G_n^p\}_n$ be as in \eqref{eq:44}, and, if $p>1$, assume that not all the zeros of $q^{\G^p}$, defined in \eqref{eq:45}, are on the boundary of the unit circle. Then, $\{\G_n^p\}_n$ is weakly well-conditioned if and only if the associated polynomial $q^{\G^p}$ has exactly two zeros of unit modulus.
\end{lemma}
\begin{proof}
Assume $p>1$, if $p=1$ the proof is analogous. According to Theorem \ref{th:41}, the matrix family $\{\G_n^p\}_n$ exhibits weakly well-conditioned, or well-conditioned, behaviour if and only if $q^{\G^p}$ falls into one of the following types: $(r_1,p+1-r_1,p-1)$ for some $r_1 \in \{0,\ldots,p+1\}$ or $(p+1,p-1-r_2,r_2)$ for some $r_2 \in \{0,\ldots,p-1\}$. It is worth noting that if $\xi$ is a zero of $q^{\G^p}$, then $\xi^{-1}$ is also a zero. Hence, $q^{\G^p}$ necessarily is a polynomial of type $(r,2p-2r,r)$ for some $r \in \{1,\ldots,p\}$. Consequently, the only admissible choice for weakly well-conditioning is $q^{\G^p}$ of type $(p-1,2,p-1)$.
\end{proof}
We aim to investigate the distribution of zeros of the polynomial $q^{\G^p}$. This polynomial belongs to the category of self-reciprocal polynomials (refer to \cite{Lakatos2002}). In other words, $q^{\G^p}$ can be expressed as $q^{\G^p}(z) = \sum_{j=0}^{2p} \g_j z^j$, where $\g_j \in \R$ and it exhibits symmetry $\g_j = \g_{2p-j}$ for $j=0,\dots, p$. These polynomials can be represented as 
\begin{equation*}
	q^{\G^p}(z) = z^p \left[ \g_0\left(\frac{1}{z^p} + z^p\right) + \g_1 \left(\frac{1}{z^{p-1}} + z^{p-1} \right) + \ldots + \g_{p-1} \left( \frac{1}{z} + z\right) + \g_p \right].
\end{equation*}
Moreover, each term $z^{-j} + z^j$ can be expressed as a polynomial in the new variable $x(z) = z^{-1} + z$. Specifically, let $T_j$ denote the $j^{\thh}$ Chebyshev polynomial defined via the recurrence relation \cite[Section 1.2.1]{MasonHandscomb2003}
\begin{equation*}
    T_0(x) = 1, \quad T_1(x) = x, \quad T_{j+1}(x) = 2x T_j(x) - T_{j-1}(x), \quad \text{for } j \ge 1.
\end{equation*}
This recursive formula has the closed form
\begin{equation*}
	T_j(x) = \frac{(x + \sqrt{x^2-1})^j + (x - \sqrt{x^2-1})^j}{2}.
\end{equation*}
Therefore, by substituting $x(z) = z^{-1} + z$, we obtain $2T_j(x(z)/2) = z^{-j} + z^j$. Exploiting this property, we introduce the so-called Chebyshev transform $\mathcal{T} q^{\G^p} \in \mathbb{P}_p(\R)$ of the self-reciprocal polynomial $q^{\G^p}$ in \eqref{eq:45}, defined as
\begin{equation} \label{eq:46}
    \mathcal{T} q^{\G^p}(x) = 2 \g_0 T_p\left(\frac{x}{2}\right) + 2\g_1 T_{p-1}\left(\frac{x}{2}\right) + \ldots +2\g_{p-1} T_1\left(\frac{x}{2}\right) + \g_p T_0\left(\frac{x}{2}\right).
\end{equation}
By definition, we have $\mathcal{T} q^{\G^p}(x(z)) = q^{\G^p}(z)/z^p$ where $x(z) = z^{-1}+z$. Additionally, observe that the leading coefficient of $\mathcal{T} q^{\G^p}$ is $\g_0$ for $p \ge 1$. This follows from the well-known fact that the leading coefficient of $T_p(x)$ is $2^{p-1}$ for $p \ge 0$ (see \cite[Equation 1.14]{Rivlin1990}).

As noted in \cite[Theorem 3]{Vieira2019}, for every pair of roots $(\xi,\xi^{-1})$ of $q^{\G^p}(z)$ located on the boundary of unit circle, there exists a corresponding zero of the polynomial $\mathcal{T} q^{\G^p}(x)$ within the interval $[-2,2]$ on the real line. This leads us to the following conclusion.
\begin{lemma} \label{lem:48}
If $p>1$, assume that not all the zeros of $q^{\G^p}$ in \eqref{eq:45} are on the boundary of the unit circle. Then, the family of matrices $\{ \G_n^p \}_n$ as defined in \eqref{eq:44} exhibits weakly well-conditioned behaviour if and only if the Chebyshev transform $\mathcal{T} q^{\G^p}$ \eqref{eq:46} of the polynomial $q^{\G^p}$ has exactly one real root in the interval $[-2,2]$.
\end{lemma}
Let $\{x_1,\ldots,x_p\}$ be the roots of $\mathcal{T} q^{\G^p}$. A necessary condition such that only one root lies in the interval $[-2,2]$ reads as
\begin{equation} \label{eq:47}
   \prod_{j=1}^p (x_j^2-4) \le 0.
\end{equation}
This condition becomes evident when all the roots are real, and even in the case of a pair of complex conjugate roots $x$ and $\bar x$, \eqref{eq:47} holds true since
\begin{equation*}
    (x^2-4)(\bar x^2-4) = (\text{Re}(x^2)-4)^2 + \text{Im}(x^2)^2 \ge 0.
\end{equation*}
In the following, we obtain a simple criterion for evaluating the product in \eqref{eq:47}.
\begin{lemma} \label{lem:49}
Consider $q^{\G^p}(z) = \g_0 z^{2p} + \g_1 z^{2p-1} + \ldots + \g_1 z + \g_0 $, a self-reciprocal polynomial of degree $2p$. Let $\{x_1,\ldots,x_p\}$ denote the roots of its Chebyshev transform $\mathcal{T} q^{\G^p}$. Assuming $\g_0 \ne 0$, the following formulas are valid
\begin{align*}
   \prod_{j=1}^p (x_j-2) = (-1)^p \, \frac{q^{\G^p}(1)}{\g_0}, \quad \quad \quad \prod_{j=1}^p (x_j+2) = \frac{q^{\G^p}(-1)}{\g_0}.
\end{align*}
\end{lemma}
\begin{proof}
To prove the first formula, let us consider $f(x) = \mathcal{T} q^{\G^p}(x+2)$. For all roots $x_j$ of $\mathcal{T} q^{\G^p}$, we have $f(x_j-2)=\mathcal{T} q^{\G^p}(x_j)=0$. Therefore, recalling Vieta's formulas the product $(x_1-2)\dots(x_p-2)$ is equal to $(-1)^p f(0)/\g_0$. We can compute $f(0) = \mathcal{T} q^{\G^p}(2) = q^{\G^p}(1)$, the latter because $x(z) = z^{-1} + z = 2$ if and only if $z=1$. For the second formula, the proof follows a similar reasoning.
\end{proof}
We conclude collecting the results from the previous lemmas.
\begin{proposition} \label{prop:410}
If the family of matrices $\{ \G_n^p \}_n$ in \eqref{eq:44} is weakly well-con-\\ ditioned, and if $p>1$ not all zeros of the associated polynomial $q^{\G^p}$, as defined in \eqref{eq:45}, are on the boundary of the unit circle, then $q^{\G^p}$ satisfies the following condition
\begin{equation} \label{eq:48}
	(-1)^p q^{\G^p}(-1) q^{\G^p}(1) \le 0.
\end{equation}
\end{proposition}
\begin{proof}
Result follows from Lemmas \ref{lem:48} and \ref{lem:49}.
\end{proof}
In general, \eqref{eq:48} is clearly not sufficient for weakly well-conditioning. Indeed, if $q^{\G^p}(1)=0$, then this condition is always satisfied. In the next section, we will explicitly calculate the quantities $q^{\G^p}(\pm 1)$ associated to the matrices \eqref{eq:310}, and we will show that \eqref{eq:48} is also sufficient for the weakly well-conditioning in this particular case.

\section{Conditioning of matrices defined by maximal regularity splines} \label{sec:5}

In this section, we present and prove the key results of this paper. We analyze the polynomial $q^{\K^p(\rho,\delta)}$ associated with the Toeplitz band parts of the family of matrices $\{\K_n^p(\rho,\delta)\}_n$ defined in \eqref{eq:310}. We investigate the conditions on $\rho$ and $\delta$ necessary for \eqref{eq:48} to hold true (Section \ref{sec:51}). Furthermore, we establish that, within this framework, this condition not only proves necessary for the weakly well-conditioning of the family $\{\K_n^p(\rho,\delta)\}_n$, but is also sufficient (Section \ref{sec:52}).

\subsection{Necessary condition for the weakly well-conditioning of $\{\K_n^p(\rho,\delta)\}_n$} \label{sec:51}
Recalling the explicit formula \eqref{eq:311}, the polynomial $q^{\K^p(\rho,\delta)}$ is given by
\begin{equation*}
    \begin{aligned}
	   q^{\K^p(\rho,\delta)}(z) & = k_0^p(\rho,\delta) + k_1^p(\rho,\delta) z + \ldots + k_p^p(\rho,\delta) z^p  + \ldots +k_1^p (\rho,\delta) z^{2p-1} + k_0^p   (\rho,\delta) z^{2p}
        \\ &  = z^p \left(k_p^p(\rho,\delta) + \sum_{j=1}^{p} k_{p-j}^p(\rho,\delta) (z^{-j} + z^j)\right)
    \end{aligned}
\end{equation*}
where the coefficients are determined in terms of the cardinal spline $\Phi_{2p+1}$ as 
\begin{equation*}
\begin{aligned}
	k_{p-j}^p(\rho,\delta) = \partial_s^2 \Phi_{2p+1}(p+1-j)  + \rho \bigl(\Phi_{2p+1}(p+1-j) 
	 + \delta (-1)^p \partial_s^{2p} \Phi_{2p+1}(p+1-j) \bigr)
	\end{aligned}
\end{equation*}
for $j = 0, \ldots, p$.

To investigate the number of zeros of $q^{\K^p(\rho,\delta)}$ on the boundary of the unit circle, we express $q^{\K^p(\rho,\delta)}(z)$ explicitly for $z = e^{\mi\theta}$ and $\theta \in [-\pi,\pi]$.
\begin{proposition} \label{prop:51}
For $p \ge 1$ and $\theta \in [-\pi,\pi]$, let the functions $M_p, B_p, D_p : [-\pi,\pi] \to \R$ be defined as
\begin{align*}
	M_p(\theta)  = (2-2\cos \theta)^{p+1} \sum_{j \in \Z} \frac{1}{(\theta + 2j\pi)^{2p+2}}, & \quad 
 B_p(\theta)  = -(2-2\cos \theta)^{p+1} \sum_{j \in \Z} \frac{1}{(\theta + 2j\pi)^{2p}},
	\\  & \hspace{-1.5cm} D_p(\theta) = (-1)^p (2-2\cos \theta)^p.
\end{align*}
Then, the following relations hold true 
\begin{align}
    \Phi_{2p+1}(p+1) + \sum_{j=1}^{p} \Phi_{2p+1}(p+1-j) (e^{-\mi j \theta}+e^{\mi j \theta}) & = M_p(\theta), \label{eq:51}
	\\ \partial_s^2 \Phi_{2p+1}(p+1) + \sum_{j=1}^{p}  \partial_s^2 \Phi_{2p+1}(p+1-j) (e^{-\mi j \theta} + e^{\mi j \theta}) & = B_p(\theta),
	\\ \partial_s^{2p} \Phi_{2p+1}(p+1) + \sum_{j=1}^{p} \partial_s^{2p} \Phi_{2p+1}(p+1-j) (e^{-\mi j \theta} + e^{\mi j \theta})  & = D_p(\theta).
\end{align}
Therefore, we can express $q^{\K^p(\rho,\delta)}(e^{\mi \theta})$ as
\begin{equation} \label{eq:54}
	q^{\K^p(\rho,\delta)}(e^{\mi \theta}) = e^{\mi p \theta}\left( B_p(\theta) + \rho\left(M_p(\theta) + \delta (-1)^p D_p(\theta) \right) \right).
\end{equation}
\end{proposition}
\begin{proof}
There is already a proof for the explicit expressions of $M_p(\theta)$ and $B_p(\theta)$ in \cite[Lemmas 7 and 6]{GaroniManniPelosiSerraCapizzanoSpeleers2014}. For the sake of completeness, we report the details for $M_p(\theta)$, and present a different proof for $B_p(\theta)$, which is useful for calculating $D_p(\theta)$.

We recall that the Fourier transform of a cardinal spline $\Phi_p$ is
\begin{equation*}
	\widehat \Phi_p(\xi) = \left(\frac{1-e^{-\mi \xi}}{\mi \xi} \right)^{p+1}, \quad \left| \widehat \Phi_p(\xi) \right|^2 = \left(\frac{2-2\cos \xi}{\xi^2} \right)^{p+1}, \quad \quad \xi \in \R
\end{equation*}
(see \cite[Example 3.4]{Chui1992}), and satisfies the hypothesis of a Poisson summation formula \cite[Theorem 2.28]{Chui1992}, i.e., for all $\theta \in [-\pi,\pi]$
\begin{equation*}
	\sum_{j \in \Z} \left(\int_{\R} \Phi_p(s) \Phi_p(s+j) \dd s\right) e^{\mi j \theta}  = \sum_{j \in \Z} \left| \widehat{\Phi}_p(\theta+2j\pi) \right|^2.
\end{equation*}
To show \eqref{eq:51}, recalling \eqref{eq:37} and \eqref{eq:38}, we compute for $\theta \in [-\pi,\pi]$, 
\begin{align*}
	\Phi_{2p+1}(p+1) & + \sum_{j=1}^{p}\Phi_{2p+1}(p+1-j) (e^{-\mi j \theta}+e^{\mi j \theta}) = \sum_{j \in \Z} \Phi_{2p+1}(p+1-j) e^{\mi j \theta}
    \\ & = \sum_{j \in \Z} \left(\int_{\R} \Phi_p(s) \Phi_p(s+j) \dd s\right) e^{\mi j \theta} = \sum_{j \in \Z} \left| \widehat{\Phi}_p\left(\theta+2j\pi\right) \right|^2
    \\ & = (2-2\cos \theta)^{p+1} \sum_{j \in \Z} \frac{1}{(\theta+2j\pi)^{2(p+1)}}.
\end{align*}
Using the recurrence relation for derivatives of cardinal splines \cite[Theorem 4.3]{Chui1992} 
\begin{equation*}
	\partial_s^k \Phi_p(s) = \partial_s^{k-1} \Phi_{p-1}(s) - \partial_s^{k-1} \Phi_{p-1}(s-1),
\end{equation*}
and considering that $e^{\mi j \theta} - 2e^{\mi (j+1) \theta} + e^{\mi (j+2) \theta} = e^{\mi j \theta} (1-e^{\mi \theta})^2 = -e^{\mi (j+1) \theta}(2-2\cos \theta)$, we can compute
\begin{align*}
	 \partial_s^2 \Phi_{2p+1}(p+1) & + \sum_{j=1}^{p} \partial_s^2 \Phi_{2p+1}(p+1-j) (e^{-\mi j \theta}+e^{\mi j \theta}) 
  \\ &= \sum_{j \in \Z} \partial_s^2 \Phi_{2p+1}(p+1-j) e^{\mi j \theta}
    \\ & = \sum_{j \in \Z} \left(\partial_s \Phi_{2p}(p+1-j)-\partial_s \Phi_{2p}(p-j)\right) e^{\mi j \theta}
	 \\ & = \sum_{j \in \Z} \left(\Phi_{2p-1}(p+1-j)-2\Phi_{2p-1}(p-j)+\Phi_{2p-1}(p-j-1)\right)  e^{\mi j \theta}
	 \\ & = -(2-2\cos \theta)\sum_{j \in \Z} \Phi_{2p-1}(p+1-j) e^{\mi (j+1) \theta} 
     \\ & = -(2-2\cos \theta)\sum_{j \in \Z} \Phi_{2p-1}(p-j) e^{\mi j \theta}
     \\ & = -(2-2\cos \theta)^{p+1} \sum_{j \in \Z} \frac{1}{(\theta + 2j\pi)^{2p}},
\end{align*}
where in the last equality we used the same procedures as in the proof for $M_p(\theta)$. Similarly, we can iteratively make the same steps as before to show that 
\begin{align*}
	\partial_s^{2p} \Phi_{2p+1}(p+1) + \sum_{j=1}^{p} \partial_s^{2p} \Phi_{2p+1}(p+1-j) (e^{-\mi j \theta}+e^{\mi j \theta})
	& = (-1)^p (2-2\cos \theta)^p \sum_{j \in \Z} \Phi_1(1-j)e^{\mi j \theta} 
	\\ & = (-1)^p (2-2\cos \theta)^{p}.
\end{align*}
\end{proof}
\begin{remark}
For $p \ge 1$, we have obtained $B_p(\theta) = -(2-2\cos \theta) M_{p-1}(\theta)$, where we set $M_0(\theta) \equiv 1$.
\end{remark}
\begin{remark}
The statement of Proposition \ref{prop:51} exhibits a close connection to \cite[Lemma 6]{GaroniManniPelosiSerraCapizzanoSpeleers2014} and \cite[Lemma 7]{GaroniManniPelosiSerraCapizzanoSpeleers2014}, where the authors investigate the symbols of matrices associated with IgA discretizations. Specifically, the $p^{\thh}$ symbol is $e_p(\theta) = f_p(\theta) / g_p(\theta)$, as detailed in, for example, \cite[Equation 15]{EkstromFurciGaroniManniSerraCapizzano2018}, where $f_p(\theta) = -B_p(\theta)$ and $g_p(\theta) = M_p(\theta)$. For an extensive discussion on how symbols are used to derive analytical spectral properties for IgA discretization of boundary value problems, see \cite{GaroniSpeleersEkstromRealiSerraCapizzanoHughes2019} and the references provided therein.
\end{remark}
Now, we determine $q^{\K^p(\rho,\delta)}(\pm 1)$ by computing $M_p(0), B_p(0)$ and $D_p(0)$, and $M_p(\pi), B_p(\pi)$ and $D_p(\pi)$.
\begin{corollary} \label{cor:54}
For $p \ge 1$, let $M_p, B_p, D_p$ be the functions defined in Proposition \ref{prop:51}. It holds true that $M_p(0)=1$, and $B_p(0) = D_p(0) = 0$. Furthermore, we have
\begin{align*}
	M_p(\pi) = \frac{2}{\pi^{2(p+1)}} (2^{2(p+1)}-1)  \zeta(2(p+1)), & \quad  \quad 
	B_p(\pi) = -2^2M_{p-1}(\pi) = -\frac{2^3}{\pi^{2p}} (2^{2p}-1) \zeta(2p),
	\\ & \hspace{-1cm} D_p(\pi) = (-1)^p 2^{2p},
\end{align*}
where $\zeta$ is the Riemann zeta function. Therefore, we have $q^{\K^p(\rho,\delta)}(1) = \rho$ and
\begin{equation*}
\begin{aligned}
	q^{\K^p(\rho,\delta)}(-1) =  (-1)^p \Biggl( -\frac{2^3}{\pi^{2p}}(2^{2p}-1) \zeta(2p)
	 + \rho \left(\frac{2}{\pi^{2(p+1)}}(2^{2(p+1)}-1) \zeta\left(2(p+1)\right) + \delta 2^{2p}\right)\Biggr).
\end{aligned}	
\end{equation*}
\end{corollary}
\begin{proof}
The limits as $\theta \to 0$ and the values at $\theta=\pi$ are readily obtained from the explicit expressions of $M_p, B_p, D_p$ in Proposition \ref{prop:51}:
\begin{align} \label{eq:55}
	&\lim_{\theta \to 0} M_p(\theta) = 1, \quad \lim_{\theta \to 0} B_p(\theta) = 0, \quad \lim_{\theta \to 0} D_p(\theta) = 0;
\end{align}
\begin{equation*}
\begin{aligned} 
	M_p(\pi) =  2^{2(p+1)} \sum_{j \in \Z} \frac{1}{(\pi + 2j\pi)^{2(p+1)}}&, \quad B_p(\pi) = -2^{2(p+1)} \sum_{j \in \Z} \frac{1}{(\pi + 2j\pi)^{2p}}, \quad D_p(\pi)= (-1)^p 2^{2p}.
\end{aligned}
\end{equation*}
It remains to explicitly calculate the two infinite sums. We have, for $m \ge 1$,
\begin{align*}
	\sum_{j \in \Z} \frac{1}{(\pi+2j \pi)^{2m}} & = \frac{1}{\pi^{2m}} \sum_{j \in \Z} \frac{1}{(1+2j)^{2m}}
	\\ & = \frac{1}{\pi^{2m}} \sum_{j=0}^\infty \frac{1}{(1+2j)^{2m}} + \frac{1}{\pi^{2m}} \sum_{j=1}^{\infty} \frac{1}{(2j-1)^{2m}}
	\\ & =  \frac{2}{\pi^{2m}} \sum_{j=1}^{\infty} \frac{1}{(2j-1)^{2m}}
	 \\ &= \frac{1}{\pi^{2m}} \frac{(2^{2m}-1)}{2^{2m-1}} \zeta(2m),
\end{align*}
where for the last equality, we used \cite[Equation 0.233.5]{GradshteynRyzhik2015}. The expressions for $q^{\K^p(\rho,\delta)}(\pm 1)$ follow from \eqref{eq:54}.
\end{proof}
\begin{remark}
The cardinal splines satisfy the property of partition for unity (see \cite[Theorem 4.3]{Chui1992}), meaning $\sum_{j \in \Z} \Phi_p(s-j) = 1$, and consequently $\sum_{j \in \Z} \partial_s \Phi_p(s-j) = 0$. Therefore, $q^{\K^p(\rho,\delta)}(1)$ can be computed directly without Proposition \ref{prop:51} as 
\begin{equation*}
\begin{aligned}
	q^{\K^p(\rho,\delta)}(1) = \sum_{j \in \Z} \partial_s^2 \Phi_{2p+1}(p+1-j) + \rho \Biggl(\sum_{j \in \Z} \Phi_{2p+1}(p+1-j) + \delta (-1)^p \sum_{j \in \Z} \partial_s^{2p} \Phi_{2p+1}(p+1-j)\Biggr) = \rho.
\end{aligned}
\end{equation*}
\end{remark}
Recalling Proposition \ref{prop:410}, we establish a condition on the parameters $\rho$ and $\delta$ that is necessary for the weakly well-conditioning of the family of matrices $\{ \K_n^p(\rho,\delta)\}_n$, for each $p \ge 1$.
\begin{proposition} \label{prop:56}
Let $p \ge 1$ and define the quantities
\begin{equation} \label{eq:56}
\begin{aligned}
	\delta_p & = -\frac{M_p(\pi)}{(-1)^p D_p(\pi)} =  -\frac{(2^{2(p+1)}-1)}{2^{2p-1} \pi^{2(p+1)}} \zeta\left(2(p+1)\right), \quad \rho_p & = - \frac{B_p(\pi)}{M_p(\pi)} = 4 \pi^2 \frac{(2^{2p}-1)}{(2^{2(p+1)}-1)} \frac{\zeta(2p)}{\zeta(2(p+1))}
	\end{aligned}
\end{equation}
with $M_p, B_p$ and $D_p$ as in Proposition \ref{prop:51}. If $p>1$, assume that not all the zeros of $q^{\K^p(\rho,\delta)}$ are on the boundary of the unit circle. Then, the following statements are valid.
\begin{itemize}
\item If the family of matrices $\{\K_n^p(\rho,\delta)\}_n$ is weakly well-conditioned for all $\rho > 0$, then $\delta \le \delta_\rho$. 
\item If the family of matrices $\{\K^p_n(\rho,0)\}_n$ is weakly well-conditioned, then $\rho \le \rho_p$.  
\end{itemize}
\end{proposition}
\begin{proof}
The result directly follows by Proposition \ref{prop:410}, Corollaries \ref{cor:54} and Theorem \ref{th:45}. Specifically, if the family of matrices $\{\K^p_n(\rho,\delta)\}_n$ is weakly well-conditioned, then we have $(-1)^p q^{\K^p(\rho,\delta)}(-1)q^{\K^p(\rho,\delta)}(1) \le 0$, which is equivalent to 
\begin{equation*}
    B_p(\pi) + \rho(M_p(\pi) + \delta (-1)^p D_p(\pi)) \le 0
\end{equation*}
or explicitly,
\begin{equation*}
	-\frac{2^3}{\pi^{2p}}(2^{2p}-1) \zeta(2p) + \rho \left(\frac{2}{\pi^{2(p+1)}}(2^{2(p+1)}-1) \zeta\left(2(p+1)\right) + \delta 2^{2p}\right) \le 0.
\end{equation*} 
\end{proof}
\begin{remark}
As a consequence of the proof of Lemma \ref{lem:59} that follows, we will conclude that for all $p \ge 1$, $\rho>0$ and $\delta \le0$ the polynomial $q^{\K^p(\rho,\delta)}$ has at most two zeros on the boundary of the unit circle. Therefore, the related assumption of Proposition \ref{prop:56} is always satisfied.
\end{remark}
\begin{remark}
Note that $\zeta(m) \to 1$ as $m \to +\infty$. This can observed from the inequality 
\begin{equation*}
	1 \le \zeta(m) = \sum_{n=1}^\infty n^{-m} = 1 + \sum_{n=2}^\infty n^{-m} \le 1 +\int_1^\infty x^{-m} \dd x = 1+ \frac{1}{m-1}.
\end{equation*}
Consequently, we deduce $\rho_p \to \pi^2$ and $\delta_p \sim -8/\pi^2 \, \pi^{-2p} \to 0$ as $p \to \infty$. These results are consistent with those obtained in \cite[Figure 1]{FLMS2023}, where the trend $\delta_p \approx -10^{-p}$ was observed. Furthermore, the constants $\rho_1$ and $\delta_1$ are exactly the same as those obtained in \cite{SZ2019}. Some values of $\rho_p$ and $\delta_p$ are reported in Table \ref{tab:1}.
\begin{table}[h]
  \centering
  \renewcommand{\arraystretch}{1.5} 
  \begin{tabular}{|c|c|c|c|c|c|c|c|c|c|c|c|}\hline
   $p$ & 1 & 2 & 3 & 4 & 5 & 6 \\\hline
   $\rho_p$ & 12 &  10   & $\nicefrac{168}{17}$ & $\nicefrac{306}{31}$ & $\nicefrac{2349}{238}$ &      $\nicefrac{7797}{790}$ \\
   \hline 
   $\delta_p$ & $-\nicefrac{1}{12}$ & $-\nicefrac{1}{120}$ & $-\nicefrac{17}{20160}$  & $-\nicefrac{5}{58529}$&   $-\nicefrac{2}{231067}$   & $-\nicefrac{1}{1140271}$  \\
  \hline 
  \end{tabular}
  \vspace{0.1cm}
  \caption{Values of $\rho_p$ and $\delta_p$ as defined in \eqref{eq:56}.}
  \label{tab:1}
\end{table}

Moreover, for all $p \ge 1$, we have $\rho_p > \pi^2$. Indeed, this is equivalent to the inequality
\begin{equation} \label{ZETA}
	\frac{\zeta(2p+2)}{\zeta(2p)} < 4 \frac{2^{2p}-1}{2^{2(p+1)}-1},
\end{equation}
which is always true as shown in \cite[Theorem 1.1]{Qi2019} (see also \cite[Equation 2.1]{Qi2019}).
\end{remark}

\subsection{Sufficient condition for the weakly well-conditioning of $\{\K_n^p(\rho,\delta)\}_n$} \label{sec:52}
We finally establish that the conditions stated in Proposition \ref{prop:56} are also sufficient for ensuring the weakly well-conditioning of the family $\{ \K^p_n(\rho,\delta)\}_n$. This fact relies primarily on technical estimates established in \cite{DonatelliGaroniManniSerraCapizzanoSpeleers2017, EkstromFurciGaroniManniSerraCapizzano2018}, within the framework of symbols-based spectral analysis for IgA discretizations. 

In the following lemma, we investigate when $q^{\K^p(\rho,\delta)}$ has exactly two zeros on the boundary of the unit circle.
\begin{lemma} \label{lem:59}
For $p \ge 1$, let $\delta_p$ and $\rho_p$ be defined as in \eqref{eq:56}. Define the function 
\begin{equation*}
	F_p(\theta;\rho,\delta) = e^{- \mi p \theta} q^{\K^p(\rho,\delta)}(e^{\mi \theta}) = B_p(\theta) + \rho(M_p(\theta) + \delta (-1)^p D_p(\theta))
\end{equation*}
with $M_p, B_p$ and $D_p$ as in Proposition \ref{prop:51}. Then, the following statements hold true.
\begin{itemize}
	\item The equation (in terms of $\theta$) $F_p(\theta;\rho,\delta) = 0$ has exactly two simple zeros in $[-\pi,\pi]$ for all $\rho > 0$ if and only if $\delta \le \delta_p$;
	\item The equation (in terms of $\theta$) $F_p(\theta;\rho,0) = 0$ has exactly two simple zeros in $[-\pi,\pi]$ if and only if $\rho \le \rho_p$.
\end{itemize}
\end{lemma}
\begin{proof}
Let $p \ge 1$. We study the equation only in the interval $[0,\pi]$ since $F_p(\theta; \rho,\delta)$ is even in $\theta$. From \eqref{eq:55} and the definition in Proposition \ref{prop:51}, we observe that $M_p(\theta) > 0$ for all $\theta \in [0,\pi]$. Hence, the ratio $F_p(\theta;\rho,\delta)/M_p(\theta)$ is well defined for $\theta \in [0,\pi]$. Our objective is to show that for all $p \ge 1$, $\rho > 0$ and $\delta \le 0$, the function $F_p(\theta; \rho, \delta)/M_p(\theta)$ is strictly decreasing in $\theta$ for $\theta \in (0,\pi)$. We begin by computing
\begin{equation*}
	\frac{d}{d \theta} \frac{F_p(\theta;\rho,\delta)}{M_p(\theta)} = \frac{d}{d \theta} \frac{B_p(\theta)}{M_p(\theta)}  + \rho \delta (-1)^p \left( \frac{D'_p(\theta)M_p(\theta) - D_p(\theta)M_p'(\theta)}{(M_p(\theta))^2} \right).
\end{equation*}
It is shown in \cite[Theorem 2]{EkstromFurciGaroniManniSerraCapizzano2018} that $\frac{d}{d \theta}\left[B_p(\theta)/M_p(\theta)\right] < 0$ for all $p \ge 1$ and $\theta \in (0,\pi)$. Additionally, we calculate $(-1)^p D_p'(\theta) = 2p \sin \theta (2-2\cos \theta)^{p-1} > 0$ for all $\theta \in (0,\pi)$. Therefore, combining \cite[Lemma A.2.]{DonatelliGaroniManniSerraCapizzanoSpeleers2017} that tell us that $M_p'(\theta) < 0$ for all $p \ge 1$ and $\theta \in (0,\pi)$, and the fact that $M_p(\theta) > 0$, and $(-1)^p D_p(\theta) > 0$, we conclude that for all $\rho >0$ and $\delta \le 0$, $\frac{d}{d \theta} [F_p(\theta; \rho, \delta)/M_p(\theta)] < 0$ for all $\theta \in (0,\pi)$. This confirms that $F_p(\theta; \rho, \delta)/M_p(\theta)$ is strictly decreasing in $\theta$ for $\rho>0$ and $\delta \le 0$.

We can state that $F_p(\theta; \rho, \delta)$ has exactly one zero in $[0,\pi]$ if and only if $F_p(0; \rho, \delta) \ge 0$ and $F_p(\pi ; \rho, \delta) \le 0$. Both quantities have been computed in Corollary \ref{cor:54}, leading to the desired result.
\end{proof}
We conclude this section with the main result of the paper.
\begin{theorem} \label{th:510}
For $p \ge 1$, let $\delta_p$ and $\rho_p$ be the quantities defined in \eqref{eq:56}. Then, the following statements hold true assuming Assumption \ref{assu:1}.
\begin{itemize}
	\item The family of matrices $\{\K_n^p(\rho,\delta)\}_n$ is weakly well-conditioned for all $\rho > 0$ if and only if $\delta \le \delta_p$. 
	\item The family $\{\K^p_n(\rho,0)\}_n$ is weakly well-conditioned if and only if $\rho \le \rho_p$.  
\end{itemize}
\end{theorem}
\begin{proof}
The result follows from Lemmas \ref{lem:59} and \ref{lem:47}, and Theorem \ref{th:45}.
\end{proof}

\section{Alternative IgA stabilizations} \label{sec:6}

In this section, we briefly analyze other possible stabilization terms, and show that, in a certain sense, the one proposed in \cite{FLMS2023} turns out to be the most effective.

Let us consider, for $k \ge 0$ and a suitable $\delta \le 0$, the perturbed bilinear form $a_{\mu,h}^{p,k} : S_{0,\bullet,h}^p \times S_{\bullet,0,h}^p \to \R$ defined as
\begin{equation}\label{eq:61}
\begin{aligned}
    a^{p,k}_{\mu,h}(u_h,v_h) =- (\partial_t u_h, \partial_t v_h)_{L^2(0,T)} & + \mu (u_h,v_h)_{L^2(0,T)} + \mu \delta h^{2k} (\partial_t^k u_h, \partial_t^k v_h)_{L^2(0,T)}.
    \end{aligned}
\end{equation}
Clearly, when $k=p$, this coincides with the stabilized bilinear form $a_{\mu,h}^p$ defined in \eqref{eq:2.5}, and for $k \ge p+1$ the stabilizing term is null, reducing to the unstabilized bilinear form as in \eqref{eq:2.4}.
Moreover, since we are ultimately employing FEM, the consistency term should be at least of order $p+1$ to ensure optimal convergence properties in both the energy and $L^2$-norms, that is $2k \ge p+1$.
\begin{remark}
Another possibility, mantaining the symmetric structure of the numerical scheme, could have been to include the more general stabilizing term:
\begin{equation*}
	 \mu \delta h^{k_1+k_2} \left((\partial_t^{k_1} u_h, \partial_t^{k_2} v_h)_{L^2(0,T)} + (\partial_t^{k_2} u_h, \partial_t^{k_1} v_h)_{L^2(0,T)}\right),
\end{equation*}
for some $k_1, k_2 \ge 0$. However, note that if $k_1 + k _2 \equiv 0 \ (\bmod \, 2)$, then after integration by parts, this term coincides with the one added in \eqref{eq:61} with $k = (k_1+k_2)/2$. Conversely, if $k_1 + k _2 \equiv 1 \ (\bmod \, 2)$, recalling \eqref{eq:37}, this term vanishes when computed on central splines.
\end{remark}
Again for $\rho = \mu h^2$ and $n=N+p-1$, the associated linear system reads:
\begin{equation*}
    \K_{\mu,h}^{p,k}(\delta) = -\B_h^p + \mu \M_h^p + \mu \delta h^{2k} \D_h^{p,k}, \quad 	\K_n^{p,k}(\rho,\delta) = h\K_{\mu,h}^{p,k}(\delta), 
\end{equation*}
with $\M_h^p$ and $\B_h^p$ as in \eqref{eq:3.2}, and the new matrices
\begin{equation*}
	\D^{p,k}_{h}[\ell,j] = \int_0^T \partial_t^k \varphi_j^p(t) \partial_t^k \varphi^p_{\ell-1}(t) \dd t 
\end{equation*}
for $j, \ell = 1,\ldots, n= N+p-1$. Similarly as in Proposition \ref{prop:32}, this new matrix has the same structure as $\D_h^p$, with central non-vanishing elements expressed as
\begin{equation*}
	h^{2k-1} \D^{p,k}_h[\ell,\ell -1 \pm j] = (-1)^k \partial_s^{2k} \Phi_{2p+1}(p+1-j)
\end{equation*}
for $j=0,\ldots,p$ and $\ell=2p+1,\ldots,N-p$. Note also that 
\begin{equation*}
	\D_h^{p,0} = \M^p_h, \quad \D_h^{p,1} = \B^p_h \quad \text{and}  \quad \D_h^{p,p} = \D_h^p.
\end{equation*}

Proposition \ref{prop:51} and Corollary \ref{cor:54} can be readily generalized for $0 \le k \le p$. 
\begin{proposition} \label{prop:62}
For $0 \le k \le p$ and $\theta \in [-\pi,\pi]$, the following relation is valid
\begin{equation*}
	 \partial_s^{2k} \Phi_{2p+1}(p+1) + \sum_{j=1}^{p} \partial_s^{2k} \Phi_{2p+1}(p+1-j) (e^{-\mi j \theta}+e^{\mi j \theta}) = D^k_p(\theta), \label{Dtk}
\end{equation*}
with $D_p^k : [-\pi,\pi] \to \R$ defined as $D_p^k(\theta) = (-1)^k (2-2\cos \theta)^k M_{p-k}(\theta)$ with $M_p$ as in Proposition \ref{prop:51}. More explicitly,
\begin{align*}
	D_p^k(\theta) = (-1)^k (2-2\cos \theta)^{p+1} \sum_{j \in \Z} \frac{1}{(\theta + 2j\pi)^{2(p+1-k)}}.
\end{align*}
In particular, we can evaluate $D_p^0(0) = 1$ and $D_p^k(0) = 0$ for $k \ge 1$, and 
\begin{equation} \label{eq:62}
\begin{aligned}
	D_p^k(\pi) & =  (-1)^k 2^{2(k+1)} M_{p-k}(\pi) = (-1)^k 2^{2k+1} \frac{2^{2(p+1-k)}-1}{\pi^{2(p+1-k)}}\zeta\left(2(p+1-k)\right),
	\end{aligned}
\end{equation}
for $0 \le k \le p$.
\end{proposition}
\begin{remark}
Given that $\zeta(2) = \pi^2/6$, the value $D_p^p(\pi)$ is in agreement with the previously obtained $D_p(\pi) = (-1)^p 2^{2p}$.
\end{remark}
The polynomial associated with the Toeplitz band parts of the family of matrices $\{ \K_n^{p,k}(\rho,\delta)\}_n$ is expressed for $\theta \in [-\pi,\pi]$ as
\begin{equation*}
	q^{\K^{p,k}(\rho,\delta)}(e^{\mi \theta}) = e^{\mi p \theta}\left( B_p(\theta) + \rho\left(M_p(\theta) + \delta (-1)^k D_p^k(\theta) \right) \right).
\end{equation*}
Generalizing Theorem \ref{th:510}, we obtain the main result of this section.
\begin{theorem}
For $p \ge 1$ and $1 \le k \le p$, let us define
\begin{equation} \label{eq:63}
	\delta_p^k = - \frac{M_p(\pi)}{(-1)^k D_p^k(\pi)} = - \frac{1}{2^{2k} \pi^{2k}} \frac{(2^{2(p+1)}-1)}{(2^{2(p+1-k)}-1)} \frac{\zeta\left(2(p+1)\right)}{\zeta\left(2(p+1-k)\right)}.
\end{equation}
Let assume the nonsingularity of the matrices $\K_n^{p,k}(\rho,\delta)$ and that all the elements on the outer diagonals of the top-left block of size $2p \times 2p$ of  each $\K_n^{p,k}(\rho,\delta)$ are non-zeros. Then, the family $\{\K_n^{p,k}(\rho,\delta)\}_n$ is weakly well-conditioned for all $\rho > 0$ if and only $\delta \le \delta_p^k$.
\end{theorem}
\begin{proof}
Let us consider the function $F_p^k(\theta;\rho,\delta) = e^{- \mi p \theta} q^{\K^{p,k}(\rho,\delta)}(e^{\mi \theta})$. We aim to show that the equation (in terms of $\theta$) $F^k_p(\theta;\rho,\delta) = 0$ has exactly one simple zero in $[0,\pi]$ for all $\rho > 0$ if and only if $\delta \le \delta_p^k$. Establishing this fact leads to the conclusion by Lemma \ref{lem:47} and Theorem \ref{th:45}.

As in the proof of Lemma \ref{lem:59}, for all $p \ge 1$, $\rho > 0$, and $\delta < 0$, we show that the function $F_p^k(\theta; \rho, \delta)/M_p(\theta)$ is strictly decreasing in $\theta$. Explicitly, we can express it as
\begin{equation*}
	\frac{F_p^k(\theta ; \rho, \delta)}{M_p(\theta)} = \frac{B_p(\theta)}{M_p(\theta)} + \rho + \rho \delta (-1)^k \frac{D_p^k(\theta)}{M_p(\theta)}.
\end{equation*}
By virtue of \cite[Theorem 2]{EkstromFurciGaroniManniSerraCapizzano2018}, showing that the last term is a decreasing functions suffices. From Proposition \ref{prop:62}, 
\begin{equation*}
	(-1)^k \frac{D^k_p(\theta)}{M_p(\theta)} = (2-2\cos \theta)^k \frac{M_{p-k}(\theta)}{M_p(\theta)} = \prod_{j=0}^{k-1} (2-2\cos \theta) \frac{M_{p-k+j}(\theta)}{M_{p-k+j+1}(\theta)}.
\end{equation*}
Each factor here is a decreasing function (as established in \cite[Lemma A.2.]{DonatelliGaroniManniSerraCapizzanoSpeleers2017}) and positive, hence the product is decreasing as well. Consequently, $F_p^k(\theta; \rho, \delta)$ has only one simple zero in $[0,\pi]$ if and only if $F_p^k(0; \rho, \delta) \ge 0$ and $F_p^k(\pi; \rho, \delta) \le 0$. This holds true for all $\rho > 0$ if and only if $\delta \le \delta_p^k$, given that $F_p^k(0;\rho,\delta)= \rho$. The explicit expression of $\delta_p^k$ is a consequence of \eqref{eq:62}.
\end{proof}
\begin{remark}
We note two distinct trends of $\delta_p^k$ as $p \to \infty$ depending on the nature of $k$:
\begin{itemize}
	\item for $k \equiv k(p)$ such that $p+1-k(p) \to \infty$ as $p \to \infty$, then $\delta_p^{k(p)} \to - \pi^{-2k(p)}$;
	\item for $k \equiv k(p)$ such that $p+1-k(p) \to M \in \N$ as $p \to \infty$, then $\delta_p^{k(p)} \sim - C_M \pi^{-2p} \to 0$ , as $p \to \infty$ with
	\begin{equation} \label{eq:64}
		C_M = \frac{1}{\pi^{(2-2M)}(1-2^{-2M}) \zeta(2M)}.
	\end{equation}
\end{itemize}
If $k(p) = p$, as in the stabilization of \cite{FLMS2023}, i.e., $M=1$, we obtain the already mentioned constant $C_1 = 8/\pi^2$. Moreover, recalling \eqref{ZETA} we conclude that $\{C_M\}_M$ is an increasing sequence, with $C_{M+1}/C_M > \pi^2$, implying that $C_M > 8\pi^{2M-4}$.

In Figure \ref{fig:2}, we exhibit the behaviour of $|\delta_p^k|$ for different choices of $k \equiv k(p)$, and in Figure \ref{fig:3}, we report the sequence $\{C_M\}_M$.
\begin{figure}[h!]
    \begin{minipage}{0.5\textwidth}
    \captionsetup{width=.8\textwidth}
        \centering
        \includegraphics[width=0.85\linewidth]{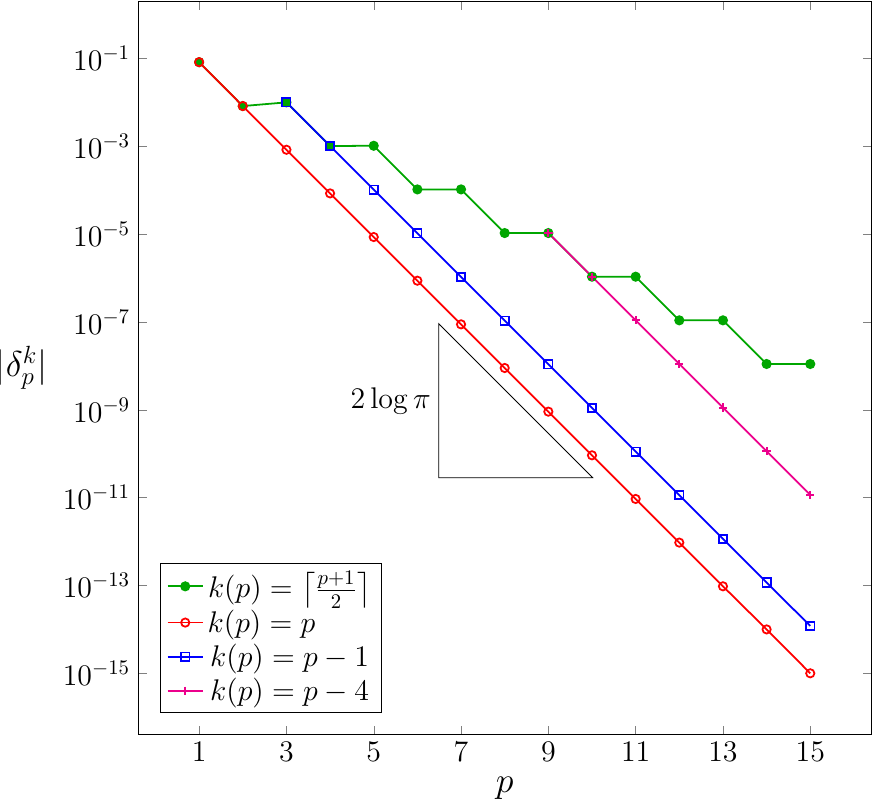}
                \caption{For various $k \equiv k(p)$, by varying $p$, the values of $|\delta_p^k|$, as in \eqref{eq:63}, on a semi-logarithmic scale.}
                \label{fig:2}
    \end{minipage}%
    \hfill
    \begin{minipage}{0.49\textwidth}
    \captionsetup{width=.8\textwidth}
        \centering
        \includegraphics[width=0.85\linewidth]{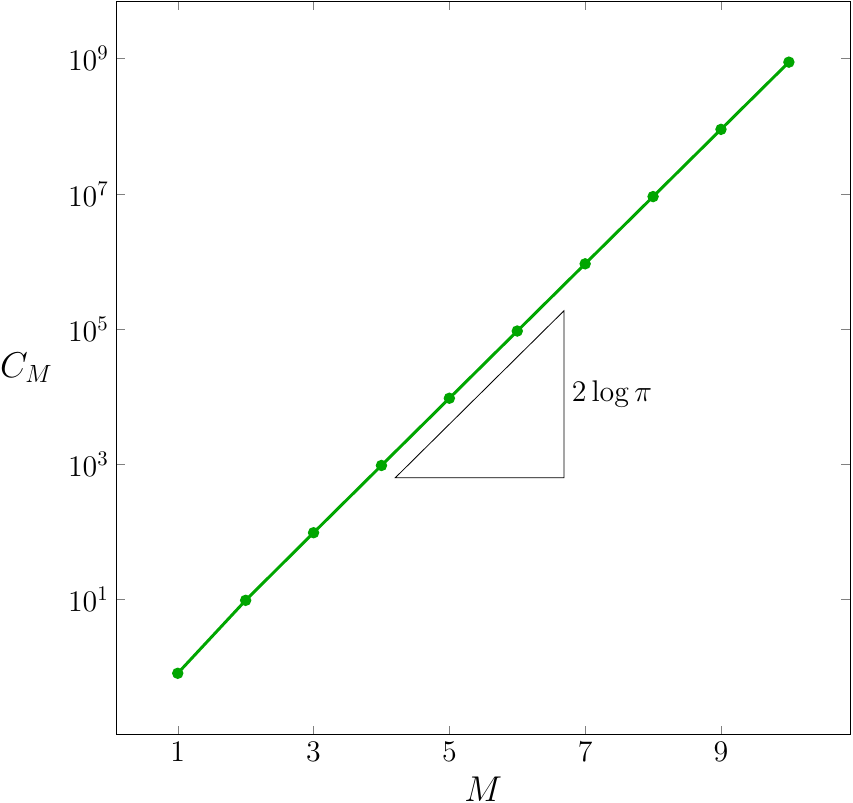}
        \caption{By varying $M$, the sequence $\{C_M\}_M$, as in \eqref{eq:64}, on a semi-logarithmic scale. \phantom{aa} \phantom{1231231}}
        \label{fig:3}
    \end{minipage}
\end{figure}
\end{remark}
We conclude by highlighting the features of the stabilized bilinear form $a_{h,\mu}^{p,p} = a_{h,\mu}^p$ (see \eqref{eq:2.5}):

\begin{itemize}
\item due to the presence of the parameter $\mu$ only as a monomial at the first power, this bilinear form can be easily adapted as a space--time bilinear form $a_h^p$ (see \eqref{eq:2.6};
\item among the possible stabilized bilinear forms $a_{h,\mu}^{p,k}$ (see \eqref{eq:61}) with $0 \le k \le p$, $a_{h,\mu}^p$ guarantees the optimal consistency error in terms of the power of $h$;
\item by considering $a_{h,\mu}^{p,k}$ with $0 \le k \le p$, and $k(p)$ a function of $p$ such that $k(p) \to p+1-M$ as $p \to \infty$ for some $M \in \N$, then the stabilized bilinear form $a_{h,\mu}^p$ guarantees the minimum magnitude for the optimal stabilizing parameter $\delta_p^k \sim -C_M \pi^{-2p}$ in terms of the constant $C_M$.
\end{itemize}

\section{Numerical results} \label{sec:7}
In this section, we provide two numerical tests to validate the sharpness of the quantities $\rho_p$ and $\delta_p$ defined in \eqref{eq:56}.

In Figures \ref{fig:4}-\ref{fig:5} we show the condition numbers of the matrices $\K_n^{p}(\rho,0)$ defined in \eqref{eq:310} with $n=1000$ and $p \in \{1,2,3\}$, and with $n=2000$ and $p \in \{4,5\}$, by varying the parameter $\rho$. With the vertical lines we also show the expected values of $\rho_p$ as detailed in Table \ref{tab:1}.

\begin{figure}[h!]
    \begin{minipage}{0.4\textwidth}
    \captionsetup{width=.9\textwidth}
        \centering
        \includegraphics[width=\linewidth]{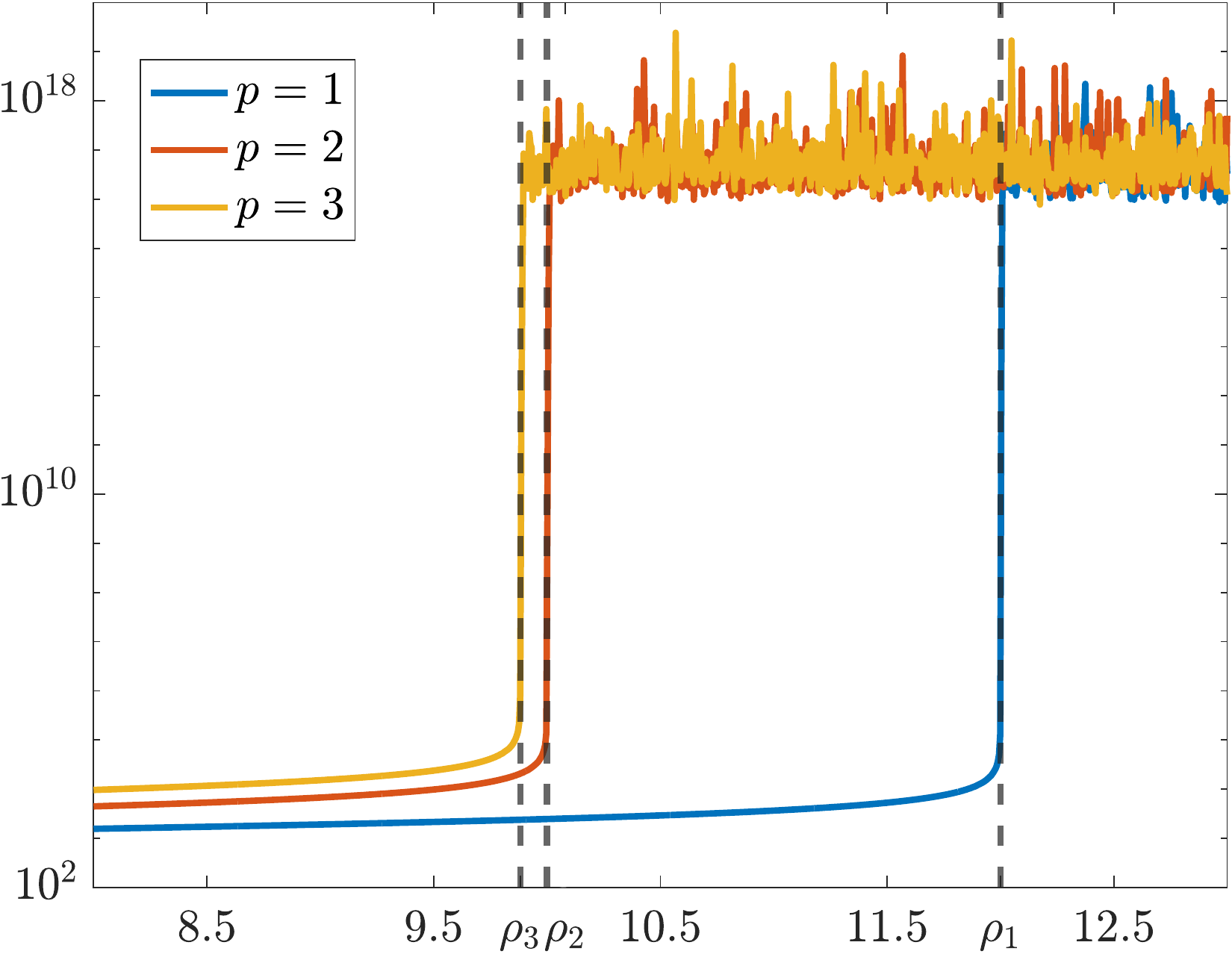}
        \caption{Spectral condition numbers $\kappa_2(\K^p_n(\rho,0))$ in semi-logarithmic scale, with $n=1000$ by varying $\rho \in [8,13]$, with $p \in \{1,2,3\}$.}
        \label{fig:4}
    \end{minipage}%
    \hfill
    \begin{minipage}{0.4\textwidth}
    \captionsetup{width=.9\textwidth}
        \centering
        \includegraphics[width=\linewidth]{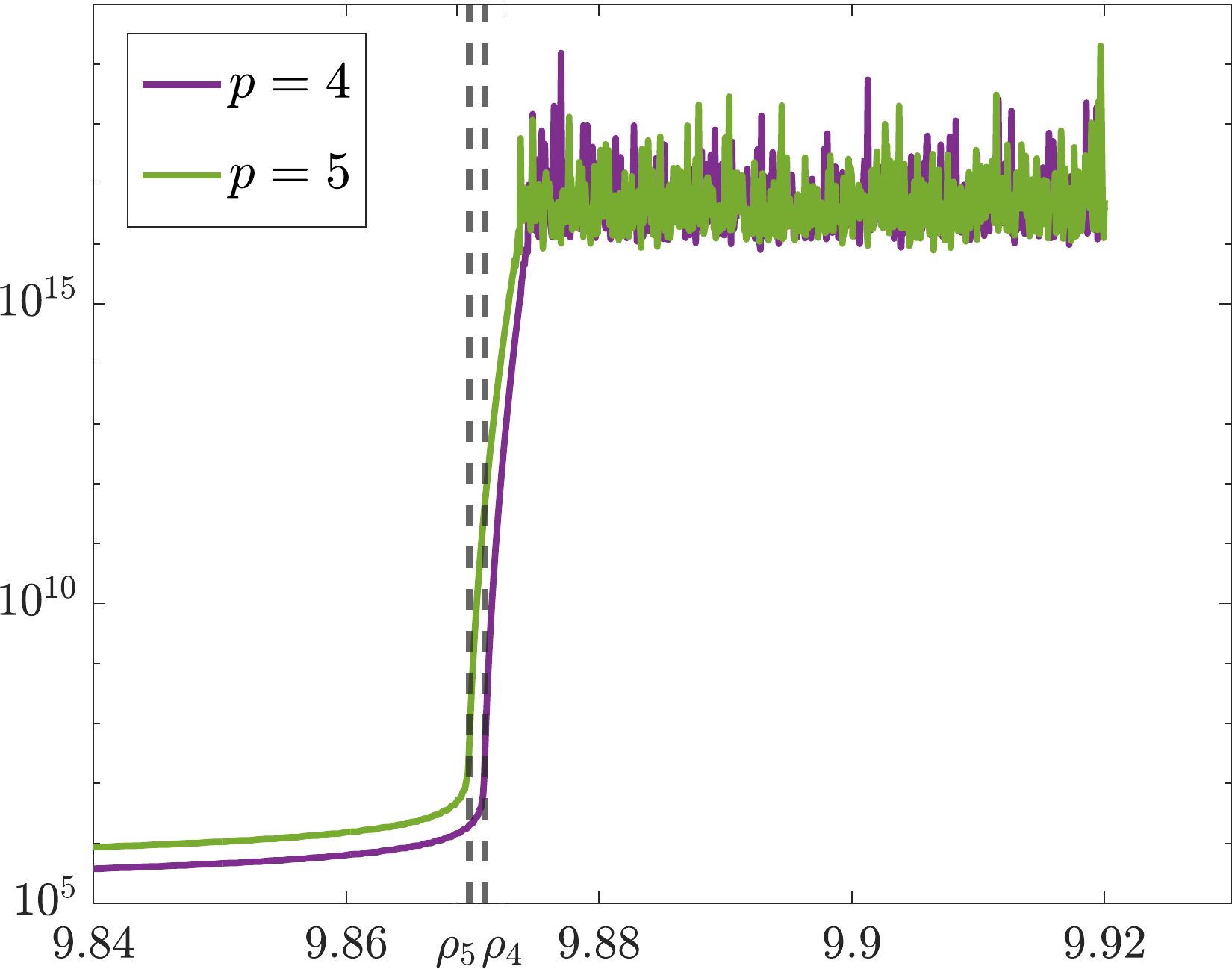}
        \caption{Spectral condition numbers $\kappa_2(\K^p_n(\rho,0))$ in semi-logarithmic scale, with $n=2000$ by varying $\rho \in [9.84,9.92]$, with $p \in \{4,5\}$.}
        \label{fig:5}
    \end{minipage}
\end{figure}

Similarly, we report in Figure \ref{fig:6} the condition numbers of the matrices $\K_{\mu,h}^p(0)$ defined in \eqref{eq:39} with $\mu = 10000$, $T=10$, and $p \in \{1,2,3\}$, by varying the parameter $h = T/N$. Recall that the CFL we expect is of the type $\mu h^2 < \rho_p$. Therefore, with the vertical lines we show the values of $\displaystyle{\sqrt{\frac{\rho_p}{\mu}}}$.

\begin{figure}[h!]
        \centering
        \includegraphics[width=0.5\linewidth]{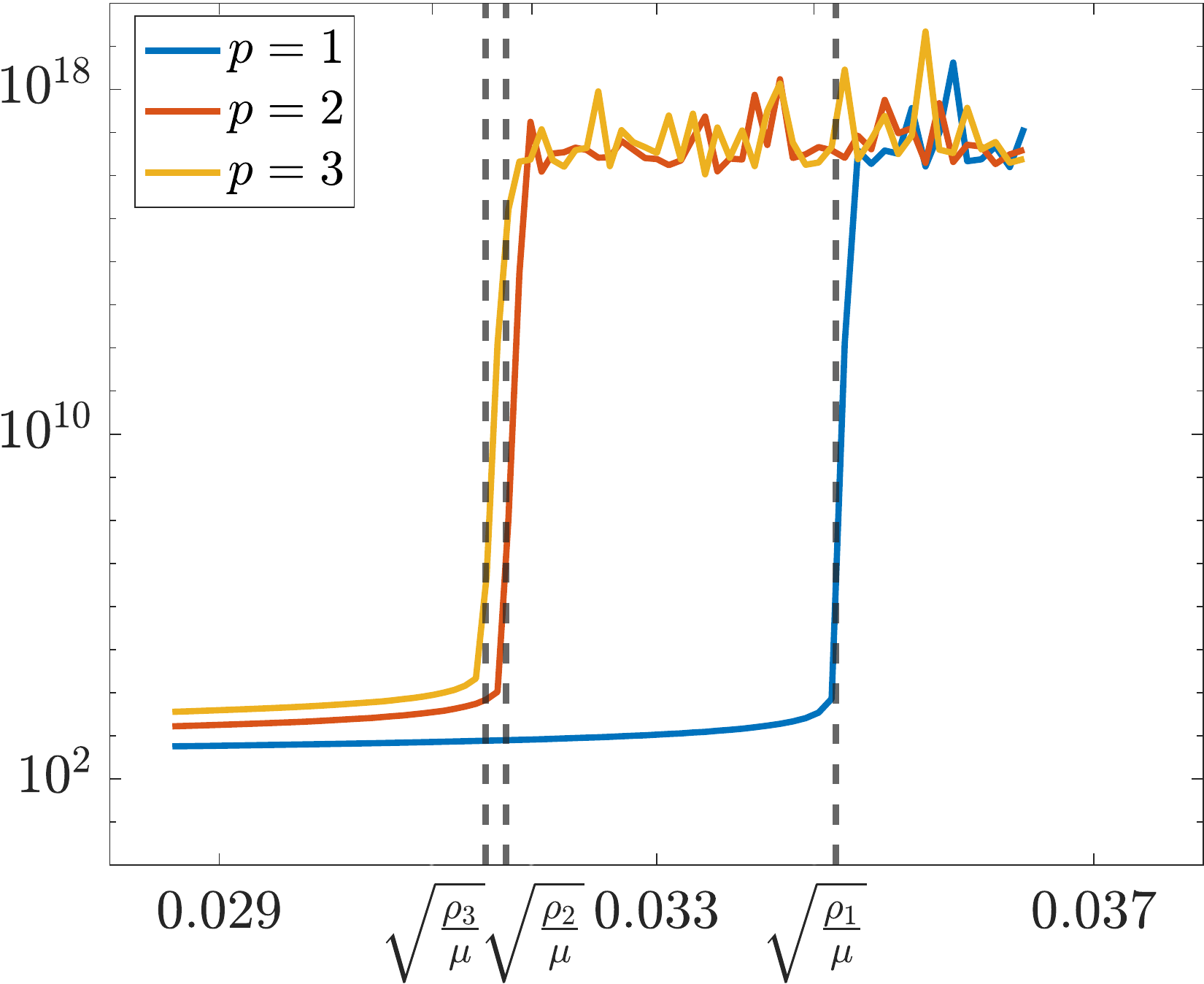}
        \caption{Spectral condition numbers $\kappa_2(\K^p_{\mu,h}(0))$ in semi-logarithmic scale, with $\mu=10000$ by varying $h  \in [0.028,0.036]$, with $p \in \{1,2,3\}$.}
        \label{fig:6}
\end{figure}

In Figure \ref{fig:7}, we show the condition numbers of the matrices $\K_n^{p}(\rho,-|\delta|)$ again with $n=1000$, $p \in \{1,2,3,4,5,6\}$, and $\rho = 20000$ by varying the absolute value $|\delta|$. We show with vertical lines the expected values $|\delta_p|$ as detailed in Table \ref{tab:1}.

\begin{figure}[h!]
    \centering
    \includegraphics[width=0.5\textwidth]{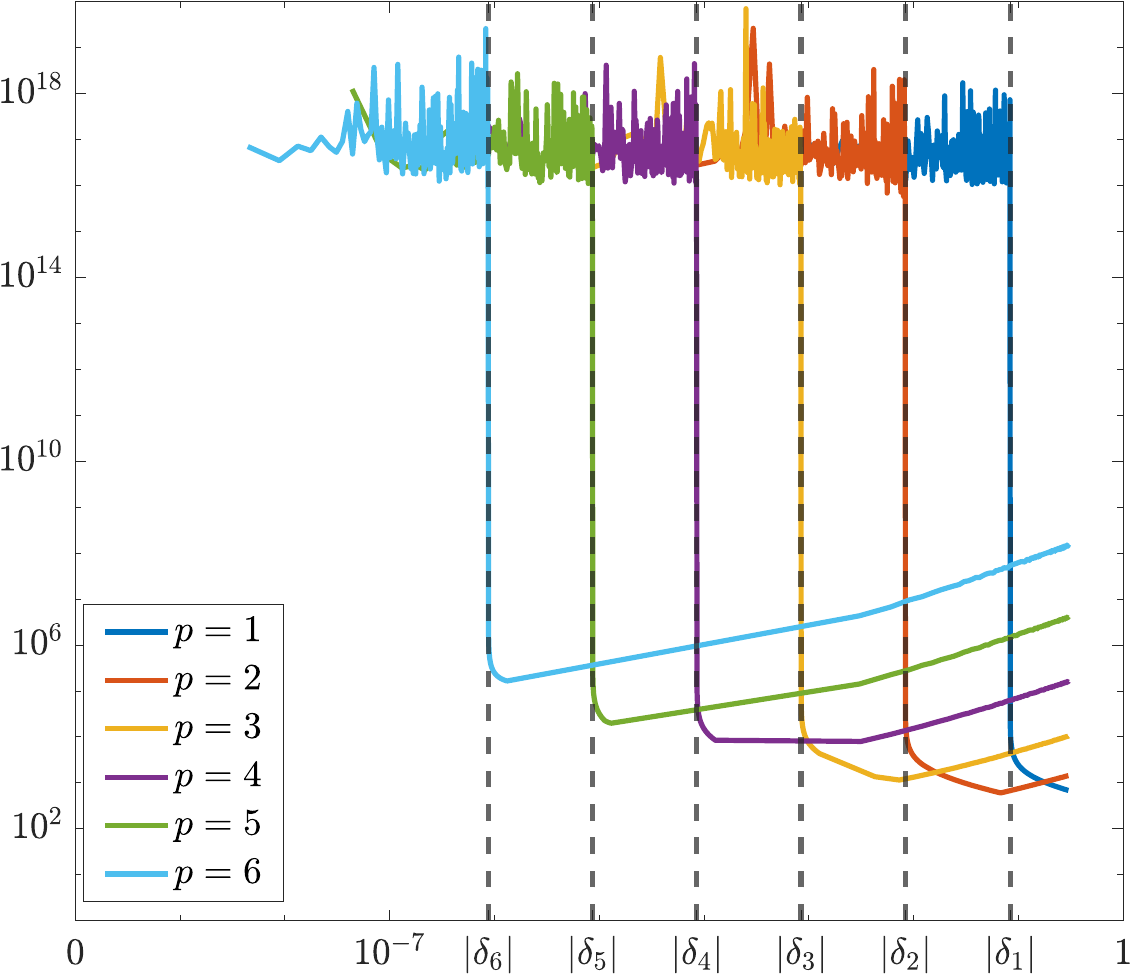}
\caption{Spectral condition numbers $\kappa_2(\K^p_n(\rho,-|\delta|))$ in logarithmic scale, with $n=1000$ and $\rho=20000$, by varying $|\delta| \in [10^{-8},0.3]$ and $p \in \{1,2,3,4,5,6\}$.}
\label{fig:7}
\end{figure}

In all cases, results of Theorem \ref{th:510} are seen to be sharp.

Finally, we report in Figures \ref{fig:8} and \ref{fig:9}, the condition numbers of the matrices $\K_{\mu,h}^p(0)$ and $\K_{\mu,h}^p(\delta_p)$ varying $h = T/N$, with $p \in \{1,2,3,4,5,6\}$, $\mu=1000$, $T=1$ and $N=2^j$, $j=7,\ldots,12$. We choose $h$ small enough to guarantee stability in both the unstabilised and stabilised cases. In both situations the condition numbers behave as $\mathcal{O}(h^2)$ and the differences between the two graphs are negligible.

\begin{figure}[h!]
    \begin{minipage}{0.4\textwidth}
    \captionsetup{width=.95\textwidth}
        \centering
        \includegraphics[width=\linewidth]{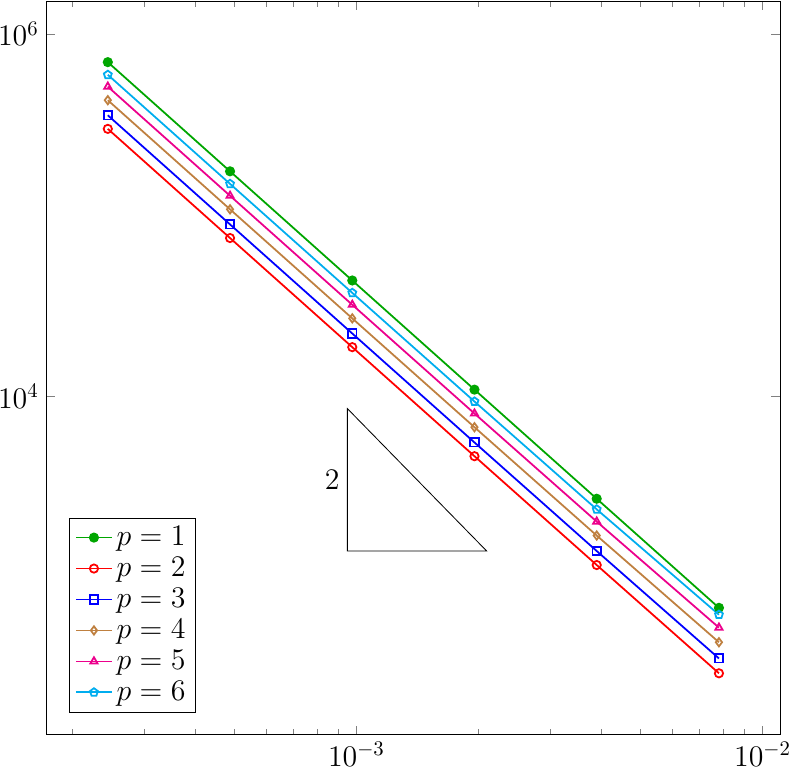}
        \caption{Spectral condition numbers $\kappa_2(\K_{\mu,h}^p(0))$ in logarithmic scale, with $\mu=10000$ and $p \in \{1,2,3,4,5,6\}$, by varying $h = T/N$ with $T=1$  and $N = 2^j$, $j=7,\ldots,12$.}
        \label{fig:8}
    \end{minipage}%
    \hfill
    \begin{minipage}{0.4\textwidth}
    \captionsetup{width=.95\textwidth}
        \centering
        \includegraphics[width=\linewidth]{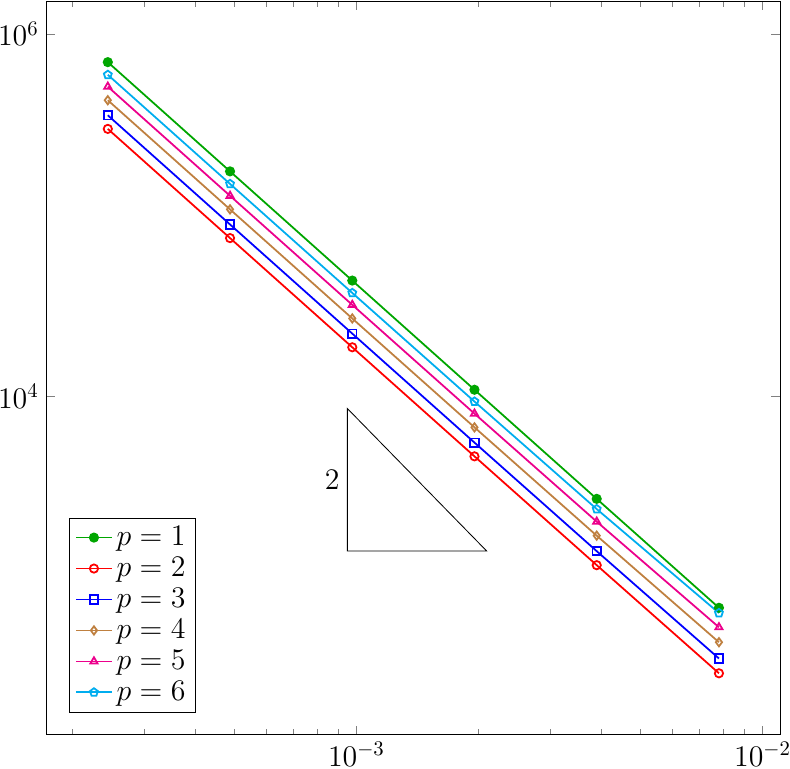}
        \caption{Spectral condition numbers $\kappa_2(\K_{\mu,h}^p(\delta_p))$ in logarithmic scale, with $\mu=10000$ and $p \in \{1,2,3,4,5,6\}$, by varying $h = T/N$ with $T=1$  and $N = 2^j$, $j=7,\ldots,12$.}
                \label{fig:9}
    \end{minipage}
\end{figure}

\section{Conclusion}
In this paper, we have examined the conditioning behaviour of a family of matrices related to conforming space--time discretizations of the wave equation, which uses splines of maximum regularity in time. We have provided a theoretical explanation of the results presented in \cite{FLMS2023}, and obtained precise estimates on two crucial parameters: the CFL condition for the unstabilized space--time bilinear form, and the optimal penalty parameter for the stabilized method. Our analysis is based on results from numerical linear algebra and properties of symbols associated with spline discretizations. Finally, we have shown how the stabilization term proposed in \cite{FLMS2023} is optimal from several points of view, and that the presented results are sharp.
\section*{Acknowledgment}
This research was supported by the Austrian Science Fund (FWF) project \href{https://doi.org/10.55776/F65}{10.55776/F65} (SF) and project \href{https://doi.org/10.55776/ P33477}{10.55776/P33477} (MF, SF). SF was also supported by the Vienna School of Mathematics. The authors would like to thank Ilaria Perugia for helpful discussions. In addition, the authors would like to thank the reviewers for their comments, which greatly improved the manuscript.

\begin{appendix}
\section{The proof of Theorem \ref{th:45}} \label{app:A}

The proof of Theorem \ref{th:45} is based on a generalization of \cite[Lemma 2]{AmodioBrugnano1996} and relies on the following lemma, which is essentially contained in \cite{AmodioBrugnano1996}.
\begin{lemma} \label{lem:A1}
Consider $\F_1, \F_2 \in \R^{k \times k}$, and $\C, \widehat{\C} \in \R^{n \times n}$ such that $\C - \widehat{\C}$ is non-zero only in the first $k$ columns. Assume the nonsingularity of the two block matrices
\begin{equation*}
    \E_{\C} = \begin{pmatrix}
    \E_{11} & \mathbf{0}_{k} \\
    \C & \E_{22}
    \end{pmatrix}_{(n+k) \times (n+k)} \quad \hspace{-2.2cm}, \quad \quad  \hspace{1.6cm} \E_{\widehat{\C}} = \begin{pmatrix}
    \E_{11} & \mathbf{0}_{k} \\
    \widehat{\C} & \E_{22}
    \end{pmatrix}_{(n+k) \times (n+k)} \hspace{-1.85cm},
\end{equation*}
with
\begin{equation*}
    \E_{11} =
    \begin{pmatrix}
        \F_1 & \mathbf{0}_{k,n-k}
    \end{pmatrix}_{k \times n}
    \quad  \quad \text{and} \quad \quad \quad 
    \E_{22} =
    \begin{pmatrix}
        \mathbf{0}_{n-k,k} \\ \F_2
    \end{pmatrix}_{n \times k}\hspace{-0.65cm},
\end{equation*}
where we denote $\mathbf{0}_{m_1,m_2} \in \R^{m_1 \times m_2}$ and $\mathbf{0}_{m_3} \in \R^{m_3 \times m_3}$ matrices with only zeros. Then, the following block decompositions apply
\begin{equation*}
    \E_{\C}^{-1} =
    \begin{pmatrix}
    \X_{11} & \X_{12} \\
    \X_{21} & \X_{22}
    \end{pmatrix}_{(n+k) \times (n+k)} \quad \hspace{-2.2cm}, \quad \hspace{1.6cm} \quad \quad 
    \E_{\widehat{\C}}^{-1} =
    \begin{pmatrix}
    \widehat{\X}_{11} & \X_{12} \\
    \widehat{\X}_{21} & \X_{22}
    \end{pmatrix}_{(n+k) \times (n+k)} \hspace{-1.8cm},
\end{equation*}
where $\X_{12} \in \R^{n \times n}$ and $\X_{21} \in \R^{k \times k}$. Moreover, $\C$ is nonsingular if and only if $\X_{21}$ is nonsingular, in which case it holds
\begin{equation} \label{eq:A1}
    \C^{-1} = \X_{12} - \X_{11} \X_{21}^{-1} \X_{22}.
\end{equation}
\end{lemma}
\begin{proof}
The first part is obtained from a direct calculation, while the second is exactly \cite[Lemma 1]{AmodioBrugnano1996}.
\end{proof}
\noindent \textit{Proof of Theorem \ref{th:45}.} The proof consists of four steps. We show in sequence that the following five matrix families have the same conditioning behaviour
\begin{equation*}
    \{\C_n\}_n = \{\C_n^{(0)}\}_n \to \{\C_n^{(1)}\}_n \to \{\C_n^{(2)}\}_n \to \{\C_n^{(3)}\}_n \to \{\C_n^{(4)}\}_n = \{\widehat{\C}_n\}_n,
\end{equation*}
where:
\begin{itemize}
\item $\{\C_n^{(1)}\}_n$ is the family of matrices differing from those in $\{\C_n^{(0)}\}_n$ only for the block of size $(m+k) \times k$ on the top-left, i.e.,
\vspace{0.5cm}
\begin{equation} \label{eq:A2}
    \C_n^{(1)} =
    \left(
    \begin{array}{c|c}
        \begin{matrix}
        \coolover{k}{\circleasterisk & \ldots & \circleasterisk} & \coolover{m}{\c_k & \phantom{**} & \phantom{***}} \\
        \vdots &  & \vdots & \vdots & \ddots \\
        \circleasterisk & \ldots & \circleasterisk & \c_{k-m+1} & \ldots & \c_k \\
        * & \ldots & * & & \\
        & \ddots & \vdots & & \\
        & & * & & 
        \end{matrix}
        &
        \begin{matrix}
        \\ & & \\
        \\ \hspace{-2.5cm} \c_k \\
        & \hspace{-1.5cm} \ddots \\
        \end{matrix}
        \\
        \hline
        \begin{matrix}
        & & & \hspace{0.7cm} \c_{-m} \\
        & & & & \hspace{0.1cm} \ddots \\
        \\ \\ &
        \end{matrix}
        &
        \begin{matrix}
        \c_0 & \ldots & \c_k & & \\
        \vdots & \ddots & & \ddots \\
        \c_{-m} & & \ddots & & \c_k \\
        & \ddots & & \ddots & \vdots \\
        &  & \c_{-m} & \ldots & \c_0 \\
        \end{matrix}
    \end{array}
    \right)_{n \times n} \hspace{-0.65cm},
\end{equation}
with the same entries of $\widehat{\C}_n$ in the top-left sub-block of size $k \times k$ from the $(m+1)^{\thh}$ to the $(m+k)^{\thh}$ rows, and whose other entries in this block (the circled ones) are such that each $(\C_n^{(1)})^{-1}$ exists for $n$ sufficiently large (we will comment in Remark \ref{rem:A5} why this is always possible),
\item $\{\C_n^{(2)}\}_n$ is the family of matrices differing from those in $\{\C_n^{(0)}\}_n$ for the whole top-left block perturbation as depicted in \eqref{eq:43}-left, i.e.,
\vspace{0.5cm}
\begin{equation*} 
    \C_n^{(2)} = 
    \left(
    \begin{array}{c|c}
        \begin{matrix}
        \coolover{m+k}{* & \ldots & * \hspace{0.2cm} & * & \phantom{**} & \phantom{**}} \\
        \vdots &  & \vdots & \vdots & \ddots & \\
        * & \ldots & * & * & \ldots & * \\
        * & \ldots & * & & & \\
        & \ddots & \vdots & & & \\
        & & * & & & 
        \end{matrix}
        &
        \begin{matrix}
        \\ & & \\
        \\ \hspace{-2.5cm} \c_k \\
        & \hspace{-1.5cm} \ddots \\
        \end{matrix}
        \\
        \hline
        \begin{matrix}
        & & & \hspace{1.2cm} \c_{-m} \\
        & & & & \hspace{-0.1cm} \ddots \\
        \\ \\ &
        \end{matrix}
        &
        \begin{matrix}
        \c_0 & \ldots & \c_k & & \\
        \vdots & \ddots & & \ddots \\
        \c_{-m} & & \ddots & & \c_k \\
        & \ddots & & \ddots & \vdots \\
        & & \c_{-m} & \ldots & \c_0 \\
        \end{matrix}
    \end{array}
    \right)_{n \times n} \hspace{-0.65cm},
\end{equation*}
with the same entries of $\widehat\C_n$,
\item $\{\C_n^{(3)}\}_n$ is the family of matrices differing from those in $\{\C_n^{(2)}\}_n$ only for the block of size $k \times (m+k)$ on the bottom-right, i.e.,
\vspace{0.5cm}
\begin{equation*}
    \C_n^{(3)} \hspace{-0.05cm} = 
    \hspace{-0.1cm}
    \left(
    \begin{array}{c|c|c}
        \hspace{-0.1cm}
        \begin{matrix}
            \coolover{m+k}{*  & \ldots & * \hspace{0.2cm}& * & \phantom{**} & \phantom{**}} \\
            \vdots & & \vdots & \vdots & \ddots & \\
            * & \ldots & * & * & \ldots & * \\
            * & \ldots & * & & & \\
            & \ddots & \vdots & & & \\
            & & * & & & 
        \end{matrix}
        &
        \begin{matrix}
            \\ & & \\
            \\ \hspace{-0.3cm} \c_k \\
            & \hspace{-0.3cm}\ddots \\
        \end{matrix}
        &
        \makebox[\wd0]{\large \phantom{1}} \\
        \hline
        \begin{matrix}
            && & \hspace{0.7cm} \c_{-m} \\
            & & & & \hspace{-0.2cm} \ddots \\
            & &
        \end{matrix}
        &
        \begin{matrix}
            \c_0 & & \hspace{0.1cm} \ddots \\
            & \ddots & \\
            \hspace{-0.1cm} \ddots & & \c_0
        \end{matrix}
        &
        \begin{matrix}
            \\ \hspace{-1.8cm} \ddots & & & \\
            & \hspace{-0.6cm} \c_k &
        \end{matrix}
        \\
        \hline
        \makebox[\wd0]{\large \phantom{1}}
        &
        \begin{matrix}
            & \hspace{0.5cm}\ddots & \\
            & & \hspace{-0.3cm} \c_{-m} \\
            & & \\
            & \\
            &
        \end{matrix}
        &
        \begin{matrix}
            & & & \c_k \\
            & & & \vdots & \ddots \\
            & & & \c_{k-m+1} & \ldots & \c_k \\
            * & \ldots & * & \circleasterisk & \ldots & \circleasterisk \\
            & \ddots & \vdots & \vdots & & \vdots \\
            & & * & \circleasterisk & \ldots & \circleasterisk
        \end{matrix}
    \end{array}
    \right)
    \hspace{-0.2cm}
    \begin{matrix}
    \phantom{\coolrightbrace{* \\  \\ \\ \\ \\ \\ \\  \vspace{0.15cm} \\ \\ \\ \\ \\  * }{k}} 
        \\ \coolrightbrace{* \\ \vspace{0.15cm} \\ * }{m} \vspace{0.05cm} \\ 
        \coolrightbrace{* \\ \vspace{0.1cm} \\ * }{k}
        \vspace{0.1cm}
    \end{matrix}
    \hspace{-0.5cm},
\end{equation*}
with the same entries of $\widehat{\C}_n$ in the bottom-right sub-block of size $k \times k$ from the $(n-m-k)^{\thh}$ to the $(n-m-1)^{\thh}$ columns, and whose other entries in this block (the circled ones) are such that each $(\C_n^{(3)})^{-1}$ exists for $n$ sufficiently large.
\end{itemize}
At each step $j=1,\ldots,4$ we show that
\begin{equation} \label{eq:A3}
    (\C_n^{(j)})^{-1} = (\C_n^{(j-1)})^{-1} + \H_n^{(j)},
\end{equation}
for four matrices $\H_n^{(j)} \in \R^{n \times n}$ such that
\begin{equation} \label{eq:A4}
    \| \H_n^{(j)}\|_{1} = \mathcal{O}\bigl(\| (\C_n^{(j-1)})^{-1}\|_{1}\bigr) \quad \quad j = 1,\ldots,4     
\end{equation}
for $n$ sufficiently large. Combining the four steps we deduce
\begin{equation} \label{eq:A5}
    \widehat{\C}_n^{-1} = \C^{-1}_n + \H_n^{(1)} + \H_n^{(2)} + \H_n^{(3)} + \H_n^{(4)},
\end{equation}
and the same conditioning behaviour for the families $\{ \C_n \}_n$ and $\{\widehat{\C}_n \}_n$.

Let $z_1,\ldots,z_{m+k}$ be the zeros of the polynomial $q^{\C}$ associated with the family $\{\C_n\}_n$, defined in \eqref{eq:42}, ordered such that $0 < | z_1 | \le \ldots \le | z_{m+k} |$. Associated with these zeros we define the matrix
\begin{equation} \label{eq:A6}
    \bZ =
    \begin{pmatrix}
        \bZ_1 &
        \\  & \bZ_2
    \end{pmatrix}_{(m+k) \times (m+k)} \hspace{-2cm},
\end{equation}
with
\begin{equation*}
    \bZ_1 =
    \begin{pmatrix}
        z_1 &
        \\ & \ddots 
        \\ & & z_m 
    \end{pmatrix}_{m \times m} \hspace{-0.85cm}, \quad \quad \quad
    \bZ_2 =
    \begin{pmatrix}
        z_{m+1} &
        \\ & \ddots 
        \\ & & z_{m+k} 
    \end{pmatrix}_{k \times k} \hspace{-0.65cm},
\end{equation*}
and the Casorati matrix $\W$ (see, e.g. \cite[Section 1.1]{AmodioBrugnano1996} and \cite[Section 2.1]{LakshmikanthamTrigiante1988} for precise definitions) partitioned as
\begin{equation*}
    \W =
    \begin{pmatrix} 
        \W_{11} & \W_{12}
        \\ \W_{21} & \W_{22}
    \end{pmatrix}_{(m+k) \times (m+k)} \hspace{-2cm},
\end{equation*}
with $\W_{12} \in \R^{k \times k}$ and $\W_{21} \in \R^{m \times m}$.

Suppose from now on the simplest case in which $q^{\C}$ is of the type $(m_1,m_2,k)$ and all its zeros are simple. In the case where $q^{\C}$ is of type $(m,k_1,k_2)$, it is sufficient to consider the families $\{\C_n^{\top}\}_n$ and $\{\widehat{\C}_n^{\top}\}_n$, while in case of non-weakly well-conditioning, or multiple zeros, proofs of \eqref{eq:A3} and \eqref{eq:A4} can be obtained by similar arguments.

\noindent\textbf{Step 1.}  Define the lower triangular block matrices
\begin{equation} \label{eq:A7}
    \E_{\C_n^{(0)}} = \begin{pmatrix}
    \E_{11} & \mathbf{0}_{k} \vspace{0.05cm} \\
    \C_n^{(0)} & \E_{22}
    \end{pmatrix}_{(n+k) \times (n+k)} \quad \hspace{-2.2cm}, \quad \quad  \hspace{1.6cm} \E_{\C_n^{(1)}} = \begin{pmatrix}
    \E_{11} & \mathbf{0}_{k} \vspace{0.05cm} \\
    \C_n^{(1)} & \E_{22}
    \end{pmatrix}_{(n+k) \times (n+k)} \hspace{-1.85cm},
\end{equation}
with
\begin{equation*}
    \E_{11} =
    \begin{pmatrix}
        \mathbf{I}_k & \mathbf{0}_{k,n-k}
    \end{pmatrix}_{k \times n}
    \quad \quad  \text{and} \quad \quad 
    \E_{22} = 
    \left(
    \begin{array}{c}
        \mathbf{0}_{n-k,k} \\[0.15cm] \hline
        \begin{matrix}
            c_k &  \\
            \vdots & \ddots  \\
            c_1 & \ldots & c_k
        \end{matrix}
    \end{array}
    \right)_{n \times k}
    \hspace{-0.65cm},
\end{equation*}
where $\mathbf{I}_k \in \R^{k \times k}$ is the identity matrix. The matrices $\E_{\C_n^{(0)}}$ and $\E_{\C_n^{(1)}}$ are nonsingular since the elements on their diagonals are not zero. Then, by virtue of Lemma \ref{lem:A1}, the following decompositions are true
\begin{equation*}
    \E_{\C_n^{(0)}}^{-1} =
    \begin{pmatrix}
    \X_{11}^{(0)} & \X_{12} \vspace{0.1cm} \\
    \X_{21}^{(0)} & \X_{22}
    \end{pmatrix}_{(n+k) \times (n+k)} \quad \hspace{-2.2cm}, \quad \hspace{1.6cm} \quad 
    \E_{\C_n^{(1)}}^{-1} =
    \begin{pmatrix}
    \X_{11}^{(1)} & \X_{12} \vspace{0.1cm} \\
    \X_{21}^{(1)} & \X_{22}
    \end{pmatrix}_{(n+k) \times (n+k)} \hspace{-1.85cm},
\end{equation*}
and with \eqref{eq:A1}, we compute
\begin{equation} \label{eq:A8}
    \H_n^{(1)} = (\C_n^{(1)})^{-1} - (\C_n^{(0)})^{-1} = \bigl(\X_{11}^{(0)} (\X_{21}^{(0)})^{-1}-\X_{11}^{(1)} (\X_{21}^{(1)})^{-1} \bigr) \X_{22}.
\end{equation}
We now show that, for $n$ sufficiently large,
\begin{equation*}
    \H_n^{(1)}[\ell,j] =  \mathcal{O}\left(\frac{|z_m|^\ell}{|z_{m+1}|^j}\right), \quad \ell = k+1,\ldots,n \quad \text{and} \quad j = 1,\ldots, n.
\end{equation*}
The first $k$ columns of $\E^{-1}_{\C_n^{(0)}}$ and $\E^{-1}_{\C_n^{(1)}}$ are obtained by solving the block linear systems
\begin{equation*}
    \E_{\C_n^{(0)}} 
    \begin{pmatrix}
        \X_{11}^{(0)} \vspace{0.1cm} \\ \X_{21}^{(0)}
    \end{pmatrix}
    = 
    \begin{pmatrix}
        \textbf{I}_k
        \\
        \mathbf{0}_{n,k}
    \end{pmatrix}, \quad \quad \quad
    \E_{\C_n^{(1)}} 
    \begin{pmatrix}
        \X_{11}^{(1)}
        \vspace{0.1cm} \\ \X_{21}^{(1)}
    \end{pmatrix}
    = 
    \begin{pmatrix}
        \textbf{I}_k
        \\
        \mathbf{0}_{n,k}
    \end{pmatrix}.
\end{equation*}
Exactly as in \cite[Equation (21)]{AmodioBrugnano1996} it is possible to show that the $\ell^{\thh}$ row of $\X_{11}^{(0)}$ satisfies the relation
\begin{equation} \label{eq:A9}
\begin{aligned}
    \begin{pmatrix}
    \X_{11}^{(0)}[\ell,1], & \ldots \, , & \X_{11}^{(0)}[\ell,k] 
    \end{pmatrix}^\top & = 
    \mathbf{e}_\ell^\top \X_{11}^{(0)} = \mathbf{1}_{m+k}^{\top} \bZ^{\ell-1}
    \begin{pmatrix}
        \G_1^{(0)} \vspace{0.1cm} \\
        \G_2^{(0)}
    \end{pmatrix}, \quad \ell = 1, \ldots, n
\end{aligned}    
\end{equation}
where $\mathbf{e}_\ell$ is the $\ell^{\thh}$ vector of the canonical base of $\R^n$, $\mathbf{1}_{m+k} = (1,\ldots,1)^{\top} \in \R^{m+k}$, $\bZ$ is defined in \eqref{eq:A6}, and the matrices $\G_1^{(0)} \in \R^{m \times k}$ and $\G_2^{(0)} \in \R^{k \times k}$ are defined by
\begin{equation*}
    \begin{pmatrix}
        \G_1^{(0)} \vspace{0.1cm} \\
        \G_2^{(0)}
    \end{pmatrix}
    = \W^{-1} \X_{11}^{(0)}[1:m+k,1:k], 
\end{equation*}
with $\W$ the Casorati matrix, and $\X_{11}^{(0)}[1:m+k,1:k]$ the $(m+k) \times k$ top-left block of $\X_{11}^{(0)}$ in standard indices notation. A similar expression is obtained for the $\ell^{\thh}$ row of $\X_{11}^{(1)}$ with the difference that the top-left block of $\C_n^{(1)}$ in \eqref{eq:A2} must be also considered, from which the analogue of \eqref{eq:A9} only holds for $\ell \ge k+1$. With greater details, we obtain
\begin{equation*}
    \mathbf{e}_\ell^\top \X_{11}^{(1)} = \mathbf{1}_{m+k}^{\top} \bZ^{\ell-1-k}
    \begin{pmatrix}
        \G_1^{(1)} \vspace{0.1cm} \\
        \G_2^{(1)}
    \end{pmatrix}, \quad \ell = k+1,\ldots, n
\end{equation*}
with the matrices $\G_1^{(1)} \in \R^{m \times k}$ and $\G_2^{(1)} \in \R^{k \times k}$ defined by
\begin{equation} \label{eq:A10}
    \begin{pmatrix}
        \G_1^{(1)} \vspace{0.1cm} \\
        \G_2^{(1)}
    \end{pmatrix}
    = \W^{-1} \X_{11}^{(1)}[k+1:m+2k,1:k].
\end{equation}
Note that the block $\X_{11}^{(1)}[k+1:m+2k,1:k]$ depends on the top-left block perturbation of size $(m+k) \times k$ of $\C_n^{(1)}$.

For $n$ sufficient large, similarly, we compute
\begin{equation} \label{eq:A11}
    \begin{aligned}
        \X_{21}^{(0)} & =
        \begin{pmatrix}
            \W_{11} & \W_{12}
        \end{pmatrix}
        \bZ^n
        \begin{pmatrix}
            \G_1^{(0)} \vspace{0.1cm} \\
            \G_2^{(0)}
        \end{pmatrix}
     = \W_{12} \bZ_2^n \G_2^{(0)} + \mathcal{O}\left(\left|\frac{z_m}{z_{m+1}}\right|^n\right),
    \\ 
        \X_{21}^{(1)} & =
        \begin{pmatrix}
            \W_{11} & \W_{12}
        \end{pmatrix}
        \bZ^{n-k}
        \begin{pmatrix}
            \G_1^{(1)} \vspace{0.1cm} \\
            \G_2^{(1)}
        \end{pmatrix}
     = \W_{12} \bZ_2^{n-k} \G_2^{(1)} + \mathcal{O}\left(\left|\frac{z_m}{z_{m+1}}\right|^{n-k}\right).
    \end{aligned}
\end{equation}
Since it is assumed that $\C_n^{(0)}$ and $\C_n^{(1)}$ are nonsingular, for sufficiently large $n$, by Lemma \ref{lem:A1} also $\G_2^{(1)}$ and $\G_2^{(0)}$ are.

Finally, $\X_{22}$ has been calculated in \cite[Equation (25)]{AmodioBrugnano1996}, and it holds true
\begin{equation} \label{eq:A12}
    \X_{22} \mathbf{e}_j =
    \begin{pmatrix}
        \W_{11} & \W_{12}
    \end{pmatrix}
    \bZ^{n+m-j}
    \begin{pmatrix}
        \G_1^{(\cdot)} \vspace{0.1cm} \\
        \G_2^{(\cdot)}
    \end{pmatrix}, \quad j = 1,\ldots,n
\end{equation}
with the vectors $\G_1^{(\cdot)}$ and $\G_2^{(\cdot)}$ not depending on $j$, defined as
\begin{equation*}
    \begin{pmatrix}
        \G_1^{(\cdot)} \vspace{0.1cm} \\
        \G_2^{(\cdot)}
    \end{pmatrix}
    = 
    \W^{-1}
    \begin{pmatrix}
        0 \\ \vdots \\ 0 \\ \c_k^{-1}
    \end{pmatrix}.
\end{equation*}
Combining \eqref{eq:A9}-\eqref{eq:A12}, for $\ell = k+1,\ldots,n$ and $j=1,\ldots,n$ we deduce from \eqref{eq:A8}, for $n$ sufficiently large,
\begin{align*}
    \H_n^{(1)}[\ell,j] & = \mathbf{e}_\ell^\top \bigl(\X_{11}^{(0)} (\X_{21}^{(0)})^{-1}-\X_{11}^{(1)} (\X_{21}^{(1)})^{-1} \bigr) \X_{22} \mathbf{e}_j 
    \\ & \hspace{-1cm} \approx \mathbf{1}_{m+k}^\top \bZ^{\ell-1}
            \left(
            \begin{pmatrix}
            \G_1^{(0)} \\
            \G_2^{(0)}
            \end{pmatrix}
            (\G_2^{(0)})^{-1}
            - \bZ^{-k}
            \begin{pmatrix}
            \G_1^{(1)} \\
            \G_2^{(1)}
            \end{pmatrix}
            (\G_2^{(1)})^{-1} \bZ_2^k
            \right) \bZ_2^{-n} \W_{12}^{-1} \X_{22} \mathbf{e}_j
    \\ & \hspace{-1cm} = \mathbf{1}_{m+k}^\top \bZ_1^{\ell-1}
            \left( 
            \G_1^{(0)} (\G_2^{(0)})^{-1}
            - \bZ_1^{-k} \G_1^{(1)} (\G_2^{(1)})^{-1} \bZ_2^k
            \right) \bZ_2^{-n} \W_{12}^{-1} \X_{22} \mathbf{e}_j
    \\ & \hspace{-1cm} \approx \mathbf{1}_{m+k}^\top \bZ_1^{\ell-1}
            \left(
            \G_1^{(0)} (\G_2^{(0)})^{-1} 
            - \bZ_1^{-k} \G_1^{(1)} (\G_2^{(1)})^{-1} \bZ_2^k
            \right) \bZ_2^{m-j} \G_2^{(\cdot)}
    \\ & \hspace{-1cm} = \mathcal{O} \left( \frac{|z_m|^\ell}{|z_{m+1}|^j} \right).
\end{align*}
From the latter, we deduce that the families $\{\C_n^{(0)}\}_n$ and $\{\C_n^{(1)}\}_n$ have the same conditioning behaviour.

\noindent\textbf{Step 2.} Consider the families of matrices $\{(\C_n^{(1)})^\top\}_n$ and $\{(\C_n^{(2)})^\top\}_n$. Since each difference $(\C_n^{(2)})^\top - (\C_n^{(1)})^\top$ has the same structure as each $\C_n^{(1)} - \C_n^{(0)}$, it is possible to proceed exactly as in \textbf{Step 1.} and bound component-wise the elements of the matrix $\H_n^{(2)}$. Note that the polynomial $q^{\C^\top}$ associated with the purely Toeplitz part of the family $\{(\C_n^{(1)})^\top\}_n$ has as roots the reciprocals of the roots of $q^{\C}$, while $m$ and $k$ now have an inverted role. Therefore, if $q^{\C}$ is of type $(m_1,m_2,k)$, then $q^{\C^\top}$ is of type $(k,m_2,m_1)$. Taking these facts into account, we deduce that, for $n$ sufficiently large,
\begin{equation*}
    \H_n^{(2)}[\ell,j] =  \mathcal{O}\left(\frac{|z_{n-k+1}|^{-j}}{|z_{n-k}|^{-\ell}}\right) = \mathcal{O}\left(\frac{|z_m|^\ell}{|z_{m+1}|^j}\right), 
\end{equation*}
for $\ell = m+1,\ldots,n$ and $j = k+1,\ldots, n$. As in the previous step, we conclude that the families $\{ \C_n^{(1)}\}_n$ and $\{ \C_n^{(2)}\}_n$ have the same conditioning behaviour.

\noindent\textbf{Step 3.} Let us define $\textbf{J}_n$ the exchange matrix of size $n \times n$, i.e.,
\begin{equation} \label{eq:A13}
    \textbf{J}_n = \begin{pmatrix}
    0 & \\
    1 & 0  \\
    0 & 1 & 0 \\
    \vdots &\ddots &\ddots & \ddots \\
    0 &\ldots& 0 & 1 & 0
    \end{pmatrix}_{n \times n} \hspace{-0.65cm}.
\end{equation}
Consider the families $\{\textbf{J}_n (\C_n^{(2)})^\top \textbf{J}_n\}_n$ and $\{\textbf{J}_n (\C_n^{(3)})^\top \textbf{J}_n \}_n$. Given a nonsingular matrix $\C \in \R^{n \times n}$, it holds valid $(\textbf{J}_n \C \textbf{J}_n)^{-1} = \textbf{J}_n \C^{-1} \textbf{J}_n$. Therefore, the conditioning behaviour of the families just defined is equal to that of $\{\C_n^{(2)}\}_n$ and $\{\C_n^{(3)}\}_n$, respectively. As in \textbf{Step 2.}, each difference $\textbf{J}_n (\C_n^{(3)})^\top \textbf{J}_n - \textbf{J}_n (\C_n^{(2)})^\top \textbf{J}_n$ has the same structure as each $\C_n^{(1)} - \C_n^{(0)}$, and it is possible to proceed as in \textbf{Step 1.} to bound component-wise the elements of $\H_n^{(3)}$. Note that in this case the polynomial $q^{\textbf{J} \C^\top \textbf{J}}$ associated with the Toeplitz part of the family $\{\textbf{J}_n (\C_n^{(2)})^\top \textbf{J}_n\}_n$ satisfies $q^{\textbf{J} \C^\top \textbf{J}} = q^{\C}$. We deduce, for $n$ sufficiently large,
\begin{equation*}
    (\textbf{J}_n (\H_n^{(3)})^\top \textbf{J}_n) [\ell,j] = \mathcal{O}\left(\frac{|z_m|^\ell}{|z_{m+1}|^j}\right), \quad \ell = k+1,\ldots,n \quad \text{and} \quad j = m+1,\ldots, n,
\end{equation*}
or equivalently
\begin{equation*}
    \H_n^{(3)} [\ell,j] = \mathcal{O}\left(\frac{|z_m|^{n-j+1}}{|z_{m+1}|^{n-\ell+1}}\right), \quad \ell = k+1,\ldots,n \quad \text{and} \quad j = m+1,\ldots, n.
\end{equation*}
We deduce the same conditioning behaviour for the families $\{ \C_n^{(2)}\}_n$ and $\{ \C_n^{(3)}\}_n$.

\noindent\textbf{Step 4.} Finally, we consider the families $\{\textbf{J}_n (\C_n^{(3)}) \textbf{J}_n\}_n$ and $\{\textbf{J}_n (\C_n^{(4)}) \textbf{J}_n \}_n$. Proceeding as in \textbf{Steps 2.} and \textbf{3.}, with estimates similar as in \textbf{Step 1.}, we obtain
\begin{equation*}
    \H_n^{(4)} [\ell,j] = \mathcal{O}\left(\frac{|z_m|^{n-j+1}}{|z_{m+1}|^{n-\ell+1}}\right), \quad \ell = m+1,\ldots,n \quad \text{and} \quad j = k+1,\ldots, n.
\end{equation*}
Combining the four steps concludes the proof.
\hfill \qedsymbol{}
\begin{remark}
Note that in Theorem \ref{th:45}, it is actually sufficient to assume that at least three of the four outer codiagonals of the two perturbed blocks have entries all different from zero.
\end{remark}
\begin{remark} \label{rem:A3}
We observe that, by a similar proof, it can be shown that Theorem \ref{th:45} also holds if instead of a perturbation of type \eqref{eq:43}, in top-left and bottom-right corners, there are two perturbation blocks of size $M_1 \times N_1$ and $M_2 \times N_2$, respectively, of type
\begin{equation*}
    \vphantom{
    \begin{matrix}
            \overbrace{XYZ}^{\mbox{$R$}} \\ \\ \\ \\ \\
            \underbrace{pqr}_{\mbox{$S$}} \\
        \end{matrix}}
    \begin{matrix}
        \coolleftbrace{M_1}{* \\ \vspace{0.15cm} \\ * \\  
        * \\ \vspace{0.15cm} \\ *}
    \end{matrix}%
    \hspace{-0.15cm}
    \begin{pmatrix}
        \coolover{N_1}{*     & \ldots &     * & * & \ldots  & * }
        \\        \vdots &        & \vdots               &         \vdots &            & \vdots
        \\         *     & \ldots &     *                &              * &  \ldots          &  *           
        \\         *     & \ldots &     *                &                &                  &                            
        \\ \vdots              &  & \vdots               &                &                  & 
        \\ \coolunder{k}{*  & \ldots &     * } & \phantom{*} & \phantom{\ldots} & \phantom{*}
        \end{pmatrix}
        \hspace{-0.15cm}
        \begin{matrix}
        \coolrightbrace{* \\ \vspace{0.15cm} \\ *}{m}
        \\\phantom{\coolrightbrace{* \\ \vspace{0.15cm} \\ *}{m}}
    \end{matrix} \hspace{-0.25cm},  \quad \quad  \quad %
         \vphantom{
    \begin{matrix}
            \overbrace{XYZ}^{\mbox{$R$}}\\ \\ \\ \\ \\ \\ \\ \\
            \underbrace{pqr}_{\mbox{$S$}} \\
        \end{matrix}}
        \hspace{-0.5cm}
        \begin{matrix}
        \phantom{\coolleftbrace{N_2}{* \\ \vspace{0.15cm} \\ *}}\vspace{-1.4cm}
        \coolleftbrace{k}{* \\ \vspace{0.15cm} \\ *} 
    \end{matrix}%
    \hspace{-0.15cm}
    \begin{pmatrix}
         &        &                &  \coolover{m}{      *  &     \ldots       & *}
        \\              &  &                     &   \vdots            &            & \vdots            
        \\              &                &                  & *     & \ldots &     *                             
        \\             *  & \ldots & *   &       *     &    \ldots            &                 *  
        \\              \vdots &   & \vdots                    &   \vdots             &              &  \vdots
        \\ \coolunder{N_2}{ *  & \ldots &     *  & * & \ldots & *} 
        \end{pmatrix}
        \hspace{-0.1cm}
    \begin{matrix}
        \coolrightbrace{* \\ \vspace{0.15cm} \\ * \\ * \\ * 
        * \\ \vspace{0.15cm} \\ * }{M_2} \vspace{0.1cm}
    \end{matrix}
    \hspace{0.1cm}.
\end{equation*}
Here, $N_1, M_1, N_1, N_2$ are independent of $n$, and at least three of the four sub-blocks of the perturbed matrices in positions $[1:m,$ $k+1:m+k]$, $[m+1:m+k$, $1:k]$, $[n-k-m+1:n-k$, $n-m+1:n]$ and $[n-k+1:n$, $n-m-k+1:n-m]$ are nonsingular.
\end{remark}
\begin{remark}
Let $\{\C_n\}_n$ be a family of nonsingular Toeplitz band matrices as in~\eqref{eq:41}. Let the associated polynomial $q^{\C}$ in \eqref{eq:42} be of type $(m_1,m_2,k)$ or $(m,k_1,k_2)$, where $m_1+m_2=m$ and $k_1+k_2=k$. From \cite[Lemma 2]{AmodioBrugnano1996}, there exist $\mathbf{L}_n, \mathbf{U}_n \in \R^{n \times n}$, respectively, lower and upper Toeplitz triangular matrices (see \cite[Section 3]{AmodioBrugnano1996} for the precise definition) such that
\begin{equation} \label{eq:A14}
    \C_n^{-1} = \mathbf{U}_n^{-1} \mathbf{L}_n^{-1} + \mathbf{H}_n^{(0)}
\end{equation}
with $\mathbf{H}_n^{(0)}$ satisfying the analogous of \eqref{eq:A4}. Consider now a family of perturbed matrices $\{\widehat{\C}_n\}_n$ as in the assumptions of Theorem \ref{th:45} or Remark \ref{rem:A3}. Combining \eqref{eq:A14} with \eqref{eq:A5}, we deduce that
\begin{equation*}
    \widehat{\C}_n^{-1} = \mathbf{U}_n^{-1} \mathbf{L}_n^{-1} + \mathbf{H}_n^{(0)} + \mathbf{H}_n^{(1)} + \mathbf{H}_n^{(2)} + \mathbf{H}_n^{(3)} + \mathbf{H}_n^{(4)}
\end{equation*}
with the same $\mathbf{L}_n$, $\mathbf{U}_n$ and $\H_n^{(0)}$ as in \cite{AmodioBrugnano1996}, and the four matrices $\H_n^{(j)}$, $j=1,\ldots,4$, characterized in the proof of Theorem \ref{th:45}.
\end{remark}
\begin{remark} \label{rem:A5}
Let the family of nonsingular Toeplitz band matrices $\{\C_n\}_n$ be weakly well-conditioned and $q^{\C}$ of type $(m_1,m_2,k)$ with $m_1+m_2=m$. Let $\{\widehat{\C}_n\}_n$ be a family of nonsingular perturbed matrices with perturbation as in the assumption of Theorem \ref{th:45}. Then, the matrices $\C_n^{(1)}$, defined in the proof of Theorem \ref{th:45}, are nonsingular, for $n$ sufficiently large, if and only if the corresponding matrix $\G_2^{(1)}$, see definition \eqref{eq:A10}, is nonsingular. This is due to Lemma \ref{lem:A1} and \eqref{eq:A11}. Consider the top-left perturbed block of $\C_n^{(1)}$ of size $(m+k) \times (m+2k)$,
\begin{equation*}
\begin{matrix}
        \coolleftbrace{m+k}{* \\ \\ \vspace{0.15cm} \\ * \\  
        * \\ \vspace{0.15cm} \\ *}
    \end{matrix}%
\begin{pmatrix}
\phantom{12} \tikzmark{left1} \circleasterisk & \ldots & \circleasterisk \hspace{0.15cm} & \tikzmark{left2} c_k & & &
\\ \phantom{12} \vdots & & \vdots \hspace{0.1cm} & \vdots & \hspace{-1.5cm} \ddots &
\\\phantom{12} \circleasterisk & \ldots & \circleasterisk \hspace{0.1cm} & \vdots &  \ddots &
\\\phantom{12} * & \ldots & * \hspace{0.1cm} & \vdots & & \ddots
\\ & \ddots & \vdots & \vdots \hspace{0.1cm} & & & \ddots 
\\ &  & * \tikzmark{right1} & \phantom{12345} c_{-m+1} & \ldots & \ldots & \ldots & c_k \tikzmark{right2} \phantom{1}
\end{pmatrix} \vspace{0.1cm}
\DrawBox[thick, black]{left1}{right1}{\textcolor{black}{\footnotesize$\mathbf{Y_1}$}}
\DrawBox[thick, black]{left2}{right2}{\textcolor{black}{\footnotesize$\mathbf{Y_2}$}} \vspace{0.2cm}
\end{equation*}
and denote as in figure the sub-blocks $\mathbf{Y}_1 \in \R^{(m+k) \times k}$ and $\mathbf{Y}_2 \in \R^{(m+k) \times (m+k)}$. Then, one can compute
\begin{equation} \label{eq:A15}
    \G_2^{(1)} = (\W^{-1})[m+1:m+k,1:m+k] \mathbf{Y}_2^{-1} \mathbf{Y}_1,
\end{equation}
where $\W$ is the Casorati matrix associated with the family $\{\C_n\}_n$. Readily follows that it is always possible to find components in $\mathbf{Y}_1$ such that $\G_2^{(1)}$ is nonsingular, and therefore, for $n$ sufficiently large, each $\C_n^{(1)}$ is nonsingular. The nonsingularity of $\C_n^{(2)}$, for $n$ sufficiently large, follows from the nonsingularity of each $\widehat{\C}_n$ and $\C_n$. Finally, with a similar idea as for $\C_n^{(1)}$, also the elements in the bottom-right block of $\{\C_n^{(3)}\}_n$ can be defined such that, for $n$ sufficiently large, each $\C_n^{(3)}$ is nonsingular. 
\end{remark}
\begin{remark} \label{rem:A6}
With a similar idea as in the previous remark, we also deduce a criterion to guarantee the nonsingularity of the elements of the family $\{\widehat \C_n\}_n$. Let the family of nonsingular Toeplitz band matrices $\{\C_n\}_n$ be weakly well-conditioned with $q^{\C}$ of type $(m_1,m_2,k)$ with $m_1+m_2=m$, and let $\{\widehat{\C}_n\}_n$ be a family of perturbed matrices with perturbation as in the assumption of Theorem \ref{th:45}. If the matrix $\G_2^{(1)}$ associated with the family $\{\widehat{\C}_n\}_n$ and $\G_2^{(1)}$ associated with $\{\mathbf{J}_n (\widehat{\C}_n)^\top \mathbf{J}_n\}_n$, being $\mathbf{J}_n$ as defined in \eqref{eq:A13}, are nonsingular, then for $n$ sufficiently large, each $\widehat\C_n$ is nonsingular. In our case, it is difficult to verify this criterion for the families $\{\K_n^p(\rho,\delta)\}_n$ since the Casorati matrix $\W$ involves the zero of the associated polynomial $q^{\K^p(\rho,\delta)}$, which are difficult to obtain due to dependence on the parameters $\rho$ and $\delta$.
\end{remark}
\begin{example} \label{ex:A7}
In this example we show that, to guarantee the nonsingularity of each $\widehat \C_n$, it is not enough that the family $\{\C_n\}_n$ is weakly well-conditioned and that the top-left and bottom-right blocks of size $(m+k) \times (m+k)$ of each element of $\{\widehat{\C}_n\}_n$ are full rank. Consider the family $\{\C_n\}_n$ with $c_{-3} = 1, c_{-2} = 2, c_{-1} = -6, c_0 = 2$ and $c_1 = 1$. The associated polynomial $q^{\C}(z) = 1 + 2z - 6z^2 + 2z^3 + z^4$ is of type $(1,2,1)$, since it has roots $-2 \pm \sqrt{3}$ and double root $1$. Therefore, the family $\{\C_n\}_n$ is weakly well-conditioned. Consider the family of perturbed matrices $\{\widehat{\C}_n\}_n$ with perturbation only in the top-left block defined as
\begin{equation*}
\widehat{\C}_n = 
\setcounter{MaxMatrixCols}{20}
    \begin{pmatrix} 
            -6 & 2  \\
            -8 & 1 & 1 \\
            1 & -6 & 2 & 1 \\
            1 & 2 & -6 & 2 & 1 \\
            & \ddots & \ddots & \ddots & \ddots & \ddots \\
             & & 1 & 2 & -6 & 2 & 1 \\
             & & & 1 & 2 & -6 & 2 \\
    \end{pmatrix}_{n \times n} \hspace{-0.5cm}.
\end{equation*}
According to the notation in the proof of Theorem \ref{th:45}, we can compute with \eqref{eq:A15}
\begin{equation*}
\G_2^{(1)} = \begin{pmatrix}
-1 & -\frac{3}{2} & - \frac{1}{2} \vspace{0.1cm}\\
-\frac{2}{3} & -1 & -\frac{1}{3} \vspace{0.1cm} \\
-\frac{23}{18} \sqrt{3} - \frac{13}{6} & -\frac{19}{12} \sqrt{3} - \frac{11}{4} & \frac{19}{36} \sqrt{3} + \frac{11}{12} \\
\end{pmatrix}
\end{equation*}
associated with the family $\{\widehat{\C}_n\}_n$. This matrix is not full rank. Therefore, for $n$ sufficiently large, each $\widehat{\C}_n$ is singular, even if the top-left block of size $4 \times 4$ of each $\widehat \C_n$ is full rank.
\end{example}

\section{The last assumption of Theorem \ref{th:45} for $\{\K_n^p(\rho,\delta)\}_n$} \label{app:B}

In this section, we explicitly show that the entries of the outer codiagonals, i.e., the $(p-1)^{\thh}$ and the $(-p-1)^{\thh}$ codiagonals, of the matrices $\K_n^p(\rho,\delta)$ are nonzero for all $p=2,\ldots,8$, $\rho>0$ and $\delta\le0$. For even $p$, we have the exception of a single value $\rho_p^*(\delta)>0$, which depends on $\delta$, and for which all the entries of these codiagonals are null. This case does not pose a problem since, in this critical case, the matrices $\K_n^p(\rho_p^*(\delta),\delta)$ are of type $m=p+1$ and $k=p-2$ and we have verified that no entries of the new outer codiagonals, i.e., the $(p-2)^{\thh}$ and $(-p)^{\thh}$ codiagonals, have null entries. Note that, thanks to the symmetry and persymmetry properties of each $\K_n^p(\rho,\delta)$ (see Proposition \ref{prop:32}), it is sufficient to analyze the first $p+1$ entries of the $(p-1)^{\thh}$ codiagonal, and in the critical case $\rho = \rho_p^*(\delta)$, the first $p+2$ entries of the $(p-2)^{\thh}$ codiagonal.

Let us denote by $\K_*^p(\rho,\delta)$ and $\K_{\#}^p(\rho,\delta)$ the vectors with the first $p+1$ and $p+2$ entries of the $(p-1)^{\thh}$ and the $(p-2)^{\thh}$ codiagonals of $\K_n^p(\rho,\delta)$, respectively.

The results of the next subsections have been performed in MATLAB 2024a with the Symbolic Toolbox \cite{matlabsymbolic}, and validated with Magma V2.28-11 \cite{BosmaCannonPlayoust1997}.

\subsection{The case $p=2$}
We explicitly obtain the entries
\begin{align*}
    \K_*^2(\rho,\delta) = 
    \left(\tfrac{1}{3}  + \rho(\tfrac{1}{60}+ 2\delta)\right)
    \cdot
    \begin{bmatrix}
    1 & \frac{1}{2} & \frac{1}{2}
    \end{bmatrix}.
\end{align*}
From which, for $\delta < -\frac{1}{120}$, there exists the positive value $$\rho_2^*(\delta) = -\tfrac{1}{3}\left(\tfrac{1}{60}+2\delta \right)^{-1}$$ such that all the elements of $\K_*^2(\rho_2^*(\delta),\delta)$ are zero. In this case,
\begin{align*}
    \K_{\#}^2(\rho_2^*(\delta),\delta) = \tfrac{4}{120\delta+1} 
    \cdot
    \begin{bmatrix}
    180\delta-1 & 30\delta-1 & 30\delta-1 & 30\delta-1
    \end{bmatrix},
\end{align*}
whose entries are nonzero if $\delta \notin \{\frac{1}{180}, \frac{1}{30}\}$.
\subsection{The case $p=3$}
We explicitly obtain the entries
\begin{align*}
    \K_*^3(\rho,\delta) = 
    \left(\tfrac{1}{20}  + \rho(\tfrac{1}{840} - 6\delta)\right)
    \cdot
    \begin{bmatrix}
    1 & \frac{1}{4} & \frac{1}{6} & \frac{1}{6}
    \end{bmatrix}.
\end{align*}
From which all the elements of $\K_*^3(\rho,\delta)$ are nonzero for all $\delta \le 0$ and $\rho>0$ .
\subsection{The case $p=4$}
We explicitly obtain the entries
\begin{align*}
    \K_*^4(\rho,\delta) = 
    \left(\tfrac{1}{210}  + \rho(\tfrac{1}{15120}+ 24\delta)\right)
    \cdot
    \begin{bmatrix}
    1 & \frac{1}{8} & \frac{1}{18} & \frac{1}{24} & \frac{1}{24}
    \end{bmatrix}.
\end{align*}
From which, for $\delta < -\frac{1}{362880}$, there exists the positive value $$\rho_4^*(\delta) = -\tfrac{1}{210}\left(\tfrac{1}{15120}+24\delta\right)^{-1}$$ such that all the elements of $\K_*^4(\rho_4^*(\delta),\delta)$ are zero. In this case, one can verify that all the entries of $\K_{\#}^4(\rho_4^*(\delta),\delta)$ are nonzero if $\delta \notin \{\tfrac{1}{127008},\tfrac{1}{119070}$\}.
\subsection{The case $p=5$}
We explicitly obtain the entries
\begin{align*}
    \K_*^5(\rho,\delta) = 
    \left(\tfrac{1}{3024}  + \rho(\tfrac{1}{3326400} - 120\delta)\right)
    \cdot
    \begin{bmatrix}
    1 & \frac{1}{16} & \frac{1}{54} & \frac{1}{96} & \frac{1}{120} & \frac{1}{120}
    \end{bmatrix}.
\end{align*}
From which all the elements of $\K_*^5(\rho,\delta)$ are nonzero  for all $\delta \le 0$ and $\rho>0$.
\subsection{The case $p=6$}
We explicitly obtain the entries
\begin{align*}
    \K_*^6(\rho,\delta) = 
    \left(\tfrac{1}{55440}  + \rho(\tfrac{1}{8648640}+ 720\delta)\right)
    \cdot
    \begin{bmatrix}
    1 & \frac{1}{32} & \frac{1}{162} & \frac{1}{384} & \frac{1}{600} & \frac{1}{720} & \frac{1}{720}
    \end{bmatrix}.
\end{align*}
From which, for $\delta < -\frac{1}{6227020800}$, there exists the positive value $$\rho_6^*(\delta) = -\tfrac{1}{55440}\left(\tfrac{1}{8648640}+720\delta\right)^{-1}$$ such that all the elements of $\K_*^6(\rho_6^*(\delta),\delta)$ are zero. In this case, one can verify that all the entries of $\K_{\#}^6(\rho_6^*(\delta),\delta)$ are nonzero if $\delta \notin \{\tfrac{1}{34248614400},\tfrac{1}{2073646575}\}$.
\subsection{The case $p=7$}
We explicitly obtain the entries
\begin{align*}
    \K_*^7(\rho,\delta) = 
    & \bigl(\tfrac{1}{1235520}  + \rho(\tfrac{1}{259459200} - 5040\delta)\bigr)
    \cdot \begin{bmatrix}
    1 & \frac{1}{64} & \frac{1}{486} & \frac{1}{1536} & \frac{1}{3000} & \frac{1}{4320} & \frac{1}{5040} & \frac{1}{5040}
    \end{bmatrix}.
\end{align*}
From which all the elements of $\K_*^7(\rho,\delta)$ are nonzero for all $\delta \le 0$ and $\rho>0$.
\subsection{The case $p=8$}
We explicitly obtain the entries
\begin{align*}
    \K_*^8(\rho,\delta) = 
    & \left(\tfrac{1}{32432400}  + \rho(\tfrac{1}{8821612800}+ 40320\delta)\right)
    \cdot \begin{bmatrix}
    1 & \frac{1}{128} & \frac{1}{1458} & \frac{1}{6144} & \frac{1}{15000} & \frac{1}{25920} & \frac{1}{35280} & \frac{1}{40320} &  \frac{1}{40320}
    \end{bmatrix}.
\end{align*}
From which, for $\delta < -\frac{1}{38109367296000}$, there exists the positive value $$\rho_8^*(\delta) = -\tfrac{1}{32432400}\left(\tfrac{1}{8821612800}+40320\delta\right)^{-1}$$ such that all the elements of $\K_*^8(\rho_8^*(\delta),\delta)$ are zero. In this case, one can verify that all the entries of $\K_{\#}^8(\rho_8^*(\delta),\delta)$ are nonzero if $\delta \notin \{\tfrac{1}{2667655710720006},\tfrac{1}{118555239552750}\}$.
\end{appendix}

\bibliography{Biblio_Ferrari_Fraschini}{}
\bibliographystyle{plain}

\end{document}